\documentclass[12pt,oneside,english]{amsart}
\usepackage[T1]{fontenc}
\usepackage[utf8]{inputenc}
\usepackage{geometry}
\usepackage{fancyhdr}
\usepackage{amsthm}
\usepackage{amstext}
\usepackage{amssymb}
\usepackage{mathdots}
\usepackage{esint}
\usepackage{xcolor}

\makeatletter
\numberwithin{equation}{section}
\numberwithin{figure}{section}
\theoremstyle{plain}
\newtheorem{thm}{\protect\theoremname}
  \theoremstyle{plain}
  \newtheorem{prop}[thm]{\protect\propositionname}
  \theoremstyle{plain}
  \newtheorem{lem}[thm]{\protect\lemmaname}
    \theoremstyle{plain}
  \newtheorem{rem}[thm]{\protect\remarkname}

  \theoremstyle{plain}
\providecommand{\tabularnewline}{\\}

\usepackage[english]{babel}
\usepackage{fancyhdr}
\providecommand{\lemmaname}{Lemma}
\providecommand{\propositionname}{Proposition}
\providecommand{\theoremname}{Theorem}
\providecommand{\corollaryname}{Corollary}
\newcommand{\abs}[1]{\ensuremath{|#1|}}

\newcommand{\Abs}[1]{\ensuremath{\left|#1\right|}}

\newcommand{\norm}[2]{\ensuremath{|\!|#1|\!|_{#2}}}

\newcommand{\Norm}[2]{\ensuremath{\left|\!\left|#1\right|\!\right|_{#2}}}

\newcommand{\braket}[2]{\langle #1 | #2 \rangle}
\newcommand{\Braket}[2]{\left\langle #1 \Big| #2 \right\rangle}

\renewcommand{\d}[1]{\ensuremath{\textnormal{d}#1}}

\newcommand{\cC}{\mathcal{C}}
\newcommand{\cD}{\mathcal{D}}

\newcommand{\cM}{\mathcal{M}}
\newcommand{\cN}{\mathcal{N}}
\newcommand{\cO}{\mathcal{O}}
\newcommand{\cP}{\mathcal{P}}

\newcommand{\cS}{\mathcal{S}}

\usepackage{babel}
\providecommand{\lemmaname}{Lemma}
  \providecommand{\propositionname}{Proposition}
  \providecommand{\remarkname}{Remark}
\providecommand{\theoremname}{Theorem}

\usepackage{babel}
\providecommand{\lemmaname}{Lemma}
  \providecommand{\propositionname}{Proposition}
\providecommand{\theoremname}{Theorem}

\makeatother

\usepackage{babel}
\providecommand{\theoremname}{Theorem}

\begin{document}

\title{A constructive approach to Schäffer's conjecture}

\author{Oleg Szehr}

\address{University of Vienna, Faculty of Mathematics, Oskar-Morgenstern-Platz
1, 1090 Vienna, Austria.}

\email{oleg.szehr@posteo.de}

\author{Rachid Zarouf}

\address{Institut de Mathématiques de Marseille, UMR 7373, Aix-Marseille Université,
39 rue F. Joliot Curie, 13453 Marseille Cedex 13, France.}

\email{rachid.zarouf@univ-amu.fr}

\address{Department of Mathematics and Mechanics, Saint Petersburg State University,
28, Universitetski pr., St. Petersburg, 198504, Russia.}


\thanks{The work is supported by Russian Science Foundation grant 14-41-00010.}
\begin{abstract}
J.J.~Schäffer proved that for {any} induced matrix norm and {any} invertible $T=T(n)$ the inequality
\[
\abs{\det{T}}\Norm{T^{-1}}{}\leq\mathcal{S}\Norm{T}{}^{n-1}
\]
holds with $\cS=\cS(n)\leq\sqrt{en}$. He conjectured that the best $\cS$ was actually bounded. This was rebutted by Gluskin-Meyer-Pajor and subsequent contributions by J.~Bourgain and H.~Queffelec that successively improved lower estimates on $\cS$.
These articles rely on a link to the theory of power sums of complex numbers. A probabilistic or number theoretic analysis of such inequalities is employed to prove the existence of $T$ with growing $\cS$ but the explicit construction of such $T$ remains an open
task. In this article we propose a constructive approach to Schäffer's
conjecture that is not related to power sum theory. As a consequence we present an explicit sequence of Toeplitz matrices with singleton spectrum $\{\lambda\}\subset\mathbb{D}-\{0\}$
such that $\cS\geq c(\lambda)\sqrt{n}$. Our framework naturally extends
to provide lower estimates on the resolvent $\Norm{(\zeta-T)^{-1}}{}$ when $\zeta\neq0$. We also obtain new upper estimates on the resolvent when the spectrum is given. This yields new upper bounds on $\Norm{T^{-1}}{}$ in terms of the eigenvalues of $T$ which slightly refine Schäffer's original estimate.\end{abstract}
\maketitle

\section{\label{sub:Introduction}Introduction}

It is a classical task in matrix analysis and operator theory to find
good estimates for inverses. A well-established line of research related
to this topic was initiated by studies of B.~L.~van der Waerden~\cite{SJ}
and J.~J.~Schäffer~\cite{SJ}: let $\cM_{n}$ be the set of complex
$n\times n$ matrices acting on the Banach space $\mathbb{C}^{n}$
equipped with a norm. What is the best $\mathcal{S}=\mathcal{S}(n)$
so that 
\[
\abs{\det{T}}\Norm{T^{-1}}{}\leq\mathcal{S}\Norm{T}{}{}^{n-1}
\]
holds for \emph{any} invertible $T\in\cM_{n}$ and \emph{any} induced
norm $\Norm{\cdot}{}=\Norm{\cdot}{\mathbb{C}^{n}\rightarrow\mathbb{C}^{n}}$?
Schäffer \cite[Theorem 3.8]{SJ} proved that 
\[
\mathcal{S}\leq\sqrt{en},
\]
and he conjectured that $\mathcal{S}$ is bounded. This claim was
refuted by lower estimates on $\mathcal{S}$ derived by E.~Gluskin,
M.~Meyer and A.~Pajor~\cite{GMP} as well as J.~Bourgain~\cite{GMP}. The currently best lower estimate is due to H.~Queffelec~\cite{QH},
where it is shown that
\[
\cS\geq\sqrt{n}(1-\cO(1/n)).
\]
The common point in the mentioned lower bounds is that they rely on an inequality of J.~Bourgain that relates Schäffer's
problem to the theory of sums of powers of complex numbers. For $\cS$ to grow the eigenvalues of $T$ should satisfy
a Turán-type power sum inequality. The construction of explicit solutions
to such inequalities appears to be a well-studied but open problem
in number theory~\cite{TP,MH,ER,AJ1,AJ2}. Accordingly the construction
of explicit sequences of matrices with growing $\cS$ remained open,
see \cite[p.~2 lines 25-26]{GMP} and \cite[final Remark p.~158]{QH}. In this article we introduce an entirely constructive approach to
Schäffer's conjecture that avoids the hard slog through power sum
theory. While previous publications focused on demonstrating the existence of \emph{spectra} (without reference to explicit matrix representations) in this article we systematically determine the optimal class of operators for the study of $\cS$. Computing explicit matrix representations we
present a sequence of Toeplitz matrices $T_{\lambda}\in\cM_{n}$ with \emph{singleton
spectrum} $\{\lambda\}\in\mathbb{D}-\{0\}$ such that
\[
\abs{\lambda}^{n}\Norm{T_{\lambda}^{-1}}{}\geq c(\lambda)\sqrt{n}\Norm{T_{\lambda}}{}{}^{n-1}.
\]
We also study upper estimates,  where we leverage on an approach established in~\cite{NFBK,NN1} to obtain new upper bounds on the resolvent $\Norm{(\zeta-T)^{-1}}{}$ of a matrix $T$ with given spectrum, with $\abs{\zeta}\leq\Norm{T}{}$. This includes the case $\zeta=0$ (Schäffer's bound) and the well-studied Davies-Simon type estimates for $\abs{\zeta}=\Norm{T}{}$~\cite{DS,NN1,ZR,SO1}. For $\zeta=0$ we derive new upper estimates on $\Norm{T^{-1}}{}$ in terms of the eigenvalues of $T$ which slightly refine Schäffer's original estimate. 

\section{Gluskin-Meyer-Pajor's approach to Schäffer's problem and Bourgain's
trick}

Gluskin-Meyer-Pajor~\cite{GMP} gave an analytic expression for $\mathcal{S}$
in terms of a \lq\lq max-min-type\rq\rq{} optimization problem,
which we shall discuss in detail in the main body of the paper. Speaking
briefly, $\mathcal{S}$ can be written as 
\begin{align}
\mathcal{S}=\sup_{\left(\lambda_{1},\dots,\lambda_{n}\right)\in\mathbb{D}^{n}}\phi\left(\lambda_{1},\dots,\lambda_{n}\right),\label{GMP_lemma}
\end{align}
where $\mathbb{D}$ is the open unit disk and $\phi$ is given by
\begin{align*}
\phi\left(\lambda_{1},\dots,\lambda_{n}\right) & :=\inf\left\{ \sum_{k=1}^{\infty}\abs{a_{k}}\:\Big|\: f(z)=\prod_{i=1}^{n}\lambda_{i}+\sum_{k=1}^{\infty}a_{k}z^{k},\: f(\lambda_{i})=0,\: i=1\dots n\right\}.
\end{align*}
As
we will discuss later $\left(\lambda_{1},\dots,\lambda_{n}\right)$
can be interpreted as the spectrum of a sequence $T=T(n)\in\cM_n$ with $\phi\left(\lambda_{1},\dots,\lambda_{n}\right)=\norm{(\det{T})\cdot{T}^{-1}}{}$.
Any choice of sequence $\left(\lambda_{1},\dots,\lambda_{n}\right)$
provides a lower bound to the supremum in~\eqref{GMP_lemma}.
Thus to show that $\mathcal{S}$ grows unboundedly Gluskin-Meyer-Pajor
employed a probabilistic method establishing the existence of a sequence
$\left(\lambda_{1},\dots,\lambda_{n}\right)$ with 
\[
\phi(\lambda_{1},....,\lambda_{n})\geq c_{1}\sqrt{\frac{n}{\log n}}\frac{1}{\log\log n},\qquad c_{1}>0.
\]
The argument was refined by a short and elegant computation of J.~Bourgain,
see~\cite[proof of Theorem 5]{GMP}, that yields 
\[
\phi(\lambda_{1},....,\lambda_{n})\geq\frac{n\prod_{i=1}^{n}\abs{\lambda_{i}}}{\max_{k\geq1}\Abs{\sum_{i=1}^{n}\lambda_{i}^{k}}}-\prod_{i=1}^{n}\abs{\lambda_{i}}.
\]
The key to lower bound this expression lies in finding
$(\lambda_{1},...,\lambda_{n})$ such that $\max_{k\geq1}\Abs{\sum_{i=1}^{n}\lambda_{i}^{k}}$
remains bounded by $\sqrt{n}$. In essence this is Turán's tenth problem, which
to date has no constructive solution~\cite{AJ1,TP}. Moreover $(\lambda_{1},...,\lambda_{n})$
must depend on $n$ or else $\prod_{i=1}^{n}\abs{\lambda_{i}}$ would
decay exponentially. Bourgain established
existence of suitable $(\lambda_{1},...,\lambda_{n})$ by a probabilistic argument and thereby proved that
\[
\phi(\lambda_{1},....,\lambda_{n})\geq c_{2}\sqrt{\frac{n}{\log n}},\qquad c_{2}>0.
\]
The currently strongest estimates are due to H.~Queffelec~\cite{QH}
and build on the above inequality by Bourgain. Queffelec uses a number
theoretic method of H.~Montgomery~\cite[Example 6, p.~101]{MH}
to approach the power sum problem: first he shows that $(\lambda_{1},....,\lambda_{n})$
can be chosen so that 
\[
\phi(\lambda_{1},....,\lambda_{n})\geq\sqrt{\frac{n}{e}},\qquad n=p-1\:\:\textnormal{and}\:\: p\:\:\textnormal{prime}.
\]
By an application of Bertrand's postulate, saying that for each $n\geq2$
there exists a prime number $p$ with $n<p<2n$, he concludes
\[
\mathcal{S}\geq\sqrt{\frac{n}{2e}},\qquad\mathcal{S}\geq\sqrt{n}(1-\cO(1/n)).
\]
Gluskin-Meyer-Pajor explicitly regret \cite[p.~2 lines 25-26]{GMP} that they were
not able to construct an example of $\left(\lambda_{1},\dots,\lambda_{n}\right)$
for which $\phi(\lambda_{1},....,\lambda_{n})$ is growing. For $p$ prime and $n=p-1$ Montgomery's example $\left(\lambda_{1},\dots,\lambda_{p-1}\right)$,
on which \cite{QH} is footed, cannot be made explicit even assuming the generalized
Riemann hypothesis \cite[final Remark p.~158]{QH}. The main contribution
of this article may be viewed as a new approach to lower estimate $\phi(\lambda_{1},....,\lambda_{n})$
that is not related to Bourgain's estimate. As a consequence we

\begin{enumerate}
\item find that the trivial choice of fixed singleton spectrum $\lambda_{1}=\dots=\lambda_{n}=\lambda\in\mathbb{D}-\{0\}$
suffices to demonstrate that $\phi(\lambda,....,\lambda)$ grows like $\sqrt{n}$ and we circumvent Turán's problem; and we
\item provide the first explicit class of counter-examples to Schäffer's conjecture:
a sequence of invertible Toeplitz matrices $T_{\lambda}\in\cM_n$ with singleton
spectrum $\{\lambda\}$ and
$$\abs{\lambda}^{n}\norm{T_{\lambda}^{-1}}{}\geq c(\lambda)\sqrt{n}\Norm{T_{\lambda}}{}^{n-1},$$
see Theorem \ref{thm:Main_theo} for details. 
\end{enumerate}
Our trick to lower bound $\phi(\lambda_{1},....,\lambda_{n})$ is
so simple that we can present it already now.
\section{\label{sec:A-constructive-method}A constructive method to lower
estimate $\cS$}

Before going into the details of our construction we present our method
to lower bound $\phi$ in the most simple language. Together with
the results of~\cite{GMP} this yields the first entirely constructive
lower estimate on $\cS$. Let $Hol(\mathbb{D})$ be the set of holomorphic
functions on $\mathbb{D}$ and let $L^{2}(\partial\mathbb{D})$ be
the usual $L^{2}$ space of the boundary $\partial\mathbb{D}$ equipped
with the standard scalar product 
\[
\left\langle f,g\right\rangle :=\int_{-\pi}^{\pi}f(e^{i\varphi})\overline{g(e^{i\varphi})}\frac{\d\varphi}{2\pi},
\]
see~\cite{NN1} for details. The Wiener algebra is the subset of
$Hol(\mathbb{D})$ of absolutely convergent Taylor series, 
\begin{align*}
W:=\{f=\sum_{k\geq0}\hat{f}(k)z^{k}|\Norm{f}{W}:=\sum_{k\geq0}\abs{\hat{f}(k)}<\infty\}.
\end{align*}
With this notation we can write the Gluskin-Meyer-Pajor expression
for $\phi$ more concisely
\begin{align}
\phi\left(\lambda_{1},\dots,\lambda_{n}\right) & =\inf\left\{ \Norm{h}{W}-\abs{h(0)}\:\Big|\: h\in W,\: h(0)=\prod_{i=1}^{n}\lambda_{i},\: h(\lambda_{i})=0,\: i=1\dots n\right\} .\label{WienerNP}
\end{align}
For $f=\sum_{k}\hat{f}(k)z^{k}$, $g=\sum_{k}\hat{g}(k)z^{k}\in Hol(\mathbb{D})$
it is well known~\cite{NN1} that the $L^{2}(\partial\mathbb{D})$
scalar product can be written as 
\[
\left\langle f,\, g\right\rangle =\sum_{j\geq0}\hat{f}(j)\overline{\hat{g}(j)}.
\]
To lower bound $\phi$ we will apply Hölder's inequality in the form
\begin{align*}
\abs{\left\langle f,\, g\right\rangle }\leq\Norm{f}{l_{\infty}^{A}}\Norm{g}{W},
\end{align*}
where $\Norm{f}{l_{\infty}^{A}}:=\sup_{k\geq0}\abs{\hat{f}(k)}$.
Finally let $B(z)=\prod_{i=1}^{n}\frac{z-\lambda_{i}}{1-\bar{\lambda}_{i}z}$ denote a finite Blaschke product. $B$ maps the unit disk onto itself and satisfies~$\overline{B(z)}=\frac{1}{B(z)},\: z\in\partial\mathbb{D}$,~\cite{GJ}.
It is well known ~\cite{NN1} that for any $h\in W$ with $h(\lambda_{i})=0$
we have $\frac{h}{B}\in W$, which is sometimes called the \lq\lq
division property\rq\rq{} of the Wiener algebra.

We are ready to lower bound $\phi$. Let $h\in W$ with $h(\lambda_{i})=0$
and $h(0)=\prod_{i=1}^{n}\lambda_{i}$ and let $g=\frac{h}{B}\in W$.
We have 
\begin{align*}
\Braket{z^{2}h}{(1-z^{2})B} & =\Braket{(z^{2}-1)h}{B}\\
 & =\Braket{(z^{2}-1)g}{1}\\
 & =-g(0)=1.
\end{align*}
Applying Hölder's inequality and observing that $\Norm{z^{2}h}{W}=\Norm{h}{W}$
we conclude that 
\begin{align*}
1 & \leq\Norm{z^{2}h}{W}\Norm{(1-z^{2})B}{l_{\infty}^{A}}\\
 & =\Norm{h}{W}\Norm{(1-z^{2})B}{l_{\infty}^{A}}.
\end{align*}
It follows that any candidate function $h$ in the definition of $\phi$
satisfies 
\[
\Norm{h}{W}\geq\frac{1}{\Norm{(1-z^{2})B}{l_{\infty}^{A}}}
\]
and consequently 
\[
\phi\left(\lambda_{1},\dots,\lambda_{n}\right)\geq\frac{1}{\Norm{(1-z^{2})B}{l_{\infty}^{A}}}-\prod_{i=1}^{n}\abs{\lambda_{i}}.
\]
This is our analogue of Bourgain's lower estimate to $\phi\left(\lambda_{1},\dots,\lambda_{n}\right)$.
It relates Schäffer's problem to a well-defined question in asymptotic
analysis. The task is to determine the asymptotic $n$-dependence of the
Fourier coefficient of $(1-z^{2})B$ with slowest decay. We have developed the tools for this in a preliminary article~\cite{SZ2}. Of course
the question about the \lq\lq right\rq\rq{} eigenvalues remains
but as we will find the trivial choice $\lambda_{1}=...=\lambda_{n}=\lambda\in(0,1)$
already reaches $\Norm{(1-z^{2})B}{l_{\infty}^{A}}\leq K(\lambda)\frac{1}{\sqrt{n}}$.
\begin{lem}
\label{lem:Blaschke_power_lemma} Given $\lambda\in(0,1)$ and $B(z)=\left(\frac{z-\lambda}{1-\lambda z}\right)^{n}$ 
there is $K=K(\lambda)>0$ such that
\[
\Norm{(1-z^{2})B}{l_{\infty}^{A}}\leq K\frac{1}{\sqrt{n}}.
\]
\end{lem}
The asymptotic analysis for the proof of the lemma is conducted in Section~\ref{sec:On-the-norm}. We conclude this section with
a take-home formulation of our constructive lower estimate on $\cS$. It is an immediate consequence of Equation~\eqref{GMP_lemma}, our lower estimate on $\phi$ and Lemma~\ref{lem:Blaschke_power_lemma}.
\begin{prop}
\label{prop1}Given any fixed $\lambda\in\mathbb{D}-\{0\}$ we have
\[
\cS\geq\phi\left(\lambda,\dots,\lambda\right)\geq c(\lambda)\sqrt{n}
\]
where $c(\lambda)>0$ depends only on $\lambda$. \end{prop}
Not only does this circumvent Turan's problem, but the estimate holds for any fixed $\lambda$. This avoids the $n$-dependence of the spectrum present in previous lower bounds.
\section{An interpolation-theoretic approach to Schäffer's question}
From now on our goal is to determine a class of \lq\lq worst\rq\rq{} operators that achieve $\phi\left(\lambda_1,\dots,\lambda_n\right)$ making thereby Proposition~\ref{prop1} entirely explicit. To this aim we start with a detailed discussion of Equation~\eqref{GMP_lemma} along the lines of Gluskin-Meyer-Pajor.

\subsection{Glusking-Meyer-Pajor's max-min problem} 
By homogeneity Schäffer's problem is to find the best~$\mathcal{S}$
such that
\[
\norm{(\det{T})\cdot{T}^{-1}}{}\leq\mathcal{S}
\]
holds for any invertible $T$ with $\norm{T}{}\leq1$. For a given
$T$ let $\cN=\cN(T)$ denote the collection of norms on $\mathbb{C}^{n}$
such that for the induced norm we have $\Norm{T}{}\leq1$. It is clear
that $\cN$ is not empty if and only if the set $\{T^{l}\ |\ l\geq0\}$
is bounded. Operators that have this property are commonly called
power bounded, i.e.~there exists a constant $C$ such that each power
of $T$ can be bounded by this constant, $\sup_{l\geq0}\Norm{T^{l}}{}\leq C$.
As a consequence $S$ can be written as a double supremum~\cite{GMP}
\begin{align}
\cS=\sup\left\{ \sup\left\{ \Norm{(\det{T})\cdot{T}^{-1}}{}\ |\ \Norm{\cdot}{}\in\cN(T)\right\} \ |\ T\ is\ power\ bounded\right\}.
\end{align}
For given $T$ the inner supremum is over all norms such that $\Norm{T}{}\leq1$.
The outer supremum is over all $T$, where power-boundedness is added
or else the inner supremum would be over an empty set. Gluskin-Meyer-Pajor
continue by proving~\cite[Proposition~2]{GMP} that if $T$ is power-bounded
then 
\begin{align*}
 \sup\left\{ \Norm{(\det{T})\cdot{T}^{-1}}{}\ |\ \Norm{\cdot}{}\in\cN(T)\right\}=\phi\left(\lambda_{1},\dots,\lambda_{n}\right).
\end{align*}
An operator on finite dimensional space is power bounded iff its spectrum is contained in the closed unit disk and no eigenvalues on the boundary
carry degeneracy. This reduces the task of lower bounding $\cS$ to the \lq\lq max-min-type\rq\rq{} optimization problem stated in Equation~\eqref{GMP_lemma}. The problem can be split into two components: $i)$ Given $(\lambda_{1},....,\lambda_{n})\in\mathbb{D}^{n}$ find the least Wiener-norm function $h$ with $h(\lambda_i)=0$ and $h(0)=1$, see Equation~\eqref{WienerNP}. This is a Nevanlinna-Pick interpolation problem in the Wiener algebra $W$. $ii)$ Find a suitable sequence $(\lambda_{1},....,\lambda_{n})\in\mathbb{D}^{n}$. The articles~\cite{GMP} and~\cite{QH} focus on the latter leaving the computation of $T$ an open task. Below we explicitly solve $i)$ using an operator-theoretic approach in terms of the norm a the so-called model operator. Computing matrix representations of this model will provide us with explicit matrices $T$ that achieve $\phi(\lambda_{1},....,\lambda_{n})$.

\subsection{Interpolation and the right class of operators} 

Let $m=\prod_{i=1}^{\abs{m}}(z-\lambda_{i})$ be a polynomial of degree
$\abs{m}$ with zeros $\lambda_{1},...,\lambda_{\abs{m}}$ in $\mathbb{D}$. The Blaschke product associated with $m$ is 
\[
B=\prod_{i=1}^{\abs{m}}b_{\lambda_{i}},\qquad b_{\lambda}=\frac{z-\lambda}{1-\bar{\lambda}z}
\]
and has numerator $m$. The $\abs{m}$-dimensional
\emph{model space} (for $W$, see \cite[p.~127]{NN1}) is defined as the quotient
vector space 
\[
K_{B}=W/BW,
\]
where $BW:=\{Bf\ |\ f\in W\}$. $W/BW$ inherits the Banach
algebra properties from $W$ and the norm on $K_{B}$ is defined as
\[
\Norm{a}{K_{B}}:=\Norm{a}{W/BW}:=\inf\{\Norm{f}{W}\:|\: f=a+mg,\, g\in W\}.
\]
We denote by $S$ the multiplication operator by $z$ on $W$ 
\begin{align*}
S:W & \rightarrow W\\
f & \mapsto S(f)=zf.
\end{align*}
The \emph{model operator} $M_{S}$ is \lq\lq{}the compression\rq\rq{}
of $S$ to the model space 
\begin{align*}
M_{S}:K_{B} & \rightarrow K_{B}\\
f & \mapsto zf.
\end{align*}
We will also use the operator norm $\Norm{M_{S}}{}:=\Norm{M_{S}}{K_{B}\rightarrow K_{B}}$.
As $K_{B}$ is an algebra it follows that multiplication by $z$ is
an operator on $W/BW$. It is known~\cite[p.~127]{NN1} that the minimal
polynomial of $M_{S}$ is equal to the numerator of $B$ and that
$\Norm{M_{S}}{}\le1$.

Interpolation problems in function algebras have been
studied in detail in the literature. For us most interesting is an
extension of the Nagy-Foiaş commutant lifting aproach to interpolation
theory~\cite{FF,NF,NF1} to general function algebras by N.K.~Nikolski~\cite[Theorem~3.4]{NN1}. For completeness the result is stated for general function algebras $A$ (see~\cite{NN1}) but we will use is only for $W$.

\begin{lem} \label{nikolai_lemma} \label{lem:Nikolai} \cite[Theorem~3.4]{NN1}
Let $m$ be a monic polynomial, $B$ the Blaschke product associated
with $m$, $A$ a function algebra and $C\geq1$. We have for $a\in A$
that 
\[
\Norm{a}{A/BA}\leq\sup\Norm{a(T)}{}\leq C\Norm{a}{A/BA},
\]
where the supremum is taken over all algebraic operators $T$ with
minimal polynomial $m$ obeying an $A$ functional calculus with constant
$C$. \end{lem} 

The proof is as simple as clever. 

\begin{proof} If $a\in A$ and $T$ admits an $A$ functional calculus
with constant $C$ then we can bound 
\[
\Norm{a(T)}{}\leq C\Norm{a}{A}.
\]
By definition $m$ is the monic polynomial of least degree such that
$m(T)=0$. Therefore we have $(a+mg)(T)=a(T)$ for any function $g\in A$.
Together with the functional calculus inequality we conclude 
\[
\Norm{a(T)}{}\leq C\inf\{\Norm{f}{A}\:|\: f=a+mg,\, g\in A\}.
\]
The inequality is achieved by the model operator $M_{S}$ acting on
$K_{B}=A/BA$. Clearly $M_{S}$ is annihilated by $m$ and obeys an
$A$ functional calculus with constant $1$. Moreover since $A$ is
a unital algebra, $\Norm{a(M_{S})}{}=\Norm{a}{A/BA}$. \end{proof}

\begin{rem} \label{Rk_ext_rational_fct} The lemma is limited to
holomorphic functions but here we are interested in inverses and resolvents.
The trick that extends the lemma to rational functions $\psi$ was
provided in~\cite{SO1}. Suppose $\psi$ has a set of poles $\{\xi_{i}\}_{i=1}^{p}$
distinct from the zeros of $m$. One can apply Lemma \ref{lem:Nikolai}
to the polynomial 
\[
a(z)=\psi\prod_{i=1}^{p}\left(\frac{m(\xi_{i})-m(z)}{m(\xi_{i})}\right),
\]
where all singularities are lifted. \end{rem}

This shows how the interpolation problem~\eqref{WienerNP} is related to the model operator $M_{S}$. We choose
$\psi=1/z$, $A=W$ and we apply Lemma \ref{lem:Nikolai} to the above
$a$. We get 
\begin{align*}
\Norm{M_{S}^{-1}}{}  =\Norm{a}{W/BW} &=\inf\left\{ \Norm{f}{W}\:|\: f\in W,\:f(\lambda_{i})=\lambda_{i}^{-1}\right\} \\
 & =\inf\left\{ \Norm{h}{W}\:|\: h\in W,\:h(\lambda_{i})=0,\: h(0)=1\right\} -1.
\end{align*}
Multiplying by $\abs{\det(M_{S})}=\prod_{i=1}^{\abs{m}}\abs{\lambda_{i}}$ and comparing to~\eqref{WienerNP}
we find the following lemma.
\begin{lem}\label{schweinchenstinkt} In the notation introduced
above we have that
\[
\Norm{\det(M_{S})M_{S}^{-1}}{}=\phi_{\abs{m}}\left(\lambda_{1},\dots,\lambda_{\abs{m}}\right)
\]
\end{lem} 
This representation provides an explicit
class of operators, which are optimal for the computation of $\phi_{\abs{m}}\left(\lambda_{1},\dots,\lambda_{\abs{m}}\right)$.
This way the inner supremum in the representation of $\cS$ is covered. The remaining supremum over sequences $(\lambda_{1},....,\lambda_{n})$
corresponds to finding a suitable sequence of eigenvalues of $M_{S}$.
The construction of such sequences is a crucial step in~\cite{GMP,QH},
where one might expect the hard work also in this article. However,
we have already seen in Proposition~\ref{prop1} that the simple choice of the
singleton spectrum $\{\lambda\}$ suffices to demonstrate that $\phi(\lambda_{1},....,\lambda_{n})$
grows like $\sqrt{n}$. From a theoretical point of view the proof of the lemma is more interesting
than its content. We wrote it out for $\psi=1/z$ but the method works
for any rational function $\psi$, which generalizes Lemma~\ref{nikolai_lemma}
to rational function. The theoretical consequence is that the Nagy-Foiaş
commutant lifting approach to interpolation theory is not limited to holomorphic functions but is also suitable for interpolation
with rational functions. See~\cite{SO1} for Lemma~\ref{lem:Nikolai} written out for
general rational functions.

\subsection{An explicit counterexample to Schäffer's conjecture}
In this section we compute the matrix entries of $M_S$ in a natural orthonormal basis for $K_B$. This makes Propostion~\ref{prop1} explicit in that it provides a sequence of power-bounded Toeplitz matrix $T_{\lambda}\in\mathcal{M}_{n}$ with $\phi(\lambda,...,\lambda)=\Norm{\det(T_\lambda)T_\lambda^{-1}}{}$.
\begin{thm}
\label{thm:Main_theo} For any fixed $\lambda\in\mathbb{D}-\{0\}$
the Toeplitz matrix

\[
T_{\lambda}=\left(\begin{array}{ccccc}
\lambda & 0 & \ldots & \ldots & 0\\
1-\lambda^{2} & \lambda & \ddots & . & \vdots\\
-\bar{\lambda}(1-\lambda^{2}) & 1-\lambda^{2} & \lambda & \ddots & \vdots\\
\vdots & \ddots & \ddots & \ddots & 0\\
(-\bar{\lambda})^{n-2}(1-\lambda^{2}) & \ldots & -\bar{\lambda}(1-\lambda^{2}) & 1-\lambda^{2} & \lambda
\end{array}\right)
\]
is an explicit counter-example to Schäffer's conjecture. $T_{\lambda}$ is power-bounded, i.e.~there exists a norm $\Norm{\cdot}{}\in\cN(T_\lambda)$ with $\Norm{T_\lambda}{}\leq1$ and
\[
\mathcal{S}\geq\Norm{\det(T_\lambda)T_{\lambda}^{-1}}{}\geq c(\lambda)\sqrt{n}.
\]
\end{thm}
In the course of the proof of the theorem we will also show the mentioned norm explicitly. The lower estimate of Section~\ref{sec:A-constructive-method} can be viewed as a simple way to lower bound $\Norm{\det(T_\lambda)T_\lambda^{-1}}{}$. The matrix $\det(T_\lambda)T_\lambda^{-1}$ is given explicitly in \cite[Theorem III.2]{SO1}. This way $\Norm{\det(T_\lambda)T_\lambda^{-1}}{}$ can in principle be computed with the support of appropriate software.
\begin{proof}
Let $H^2\subset L^2(\partial\mathbb{D})$ denote the standard Hardy space of the boundary $\partial{\mathbb{D}}$, see \cite{NN2} for details. A well-known orthonormal basis for the space $H^2\ominus BH^2$ is the
the Malmquist-Walsh basis $\{e_{j}\}_{j=1,...,\abs{m}}$
given by (\cite[p.~117]{NN3}) 
\begin{align*}
e_{j}(z):=\frac{(1-\abs{\lambda_{j}}^{2})^{1/2}}{1-\bar{\lambda}_{j}z}\prod_{i=1}^{k-1}\frac{z-\lambda_{i}}{1-\bar{\lambda}_{i}z}.
\end{align*}
The empty product is defined to be $1$ i.e.~$e_{1}(z)=\frac{(1-\abs{\lambda_{1}}^{2})^{1/2}}{1-\bar{\lambda}_{1}z}$. Making use of the fact that $W\subset H^2$ and that rational functions are contained in $W$ it is not hard to see that the equivalence classes $[e_{j}]:=\{e_j+Bf\:|\:f\in W\}$ with $j=1,...,\abs{m}$ constitute an orthonormal basis of $W/BW$ (with respect to the scalar product inherited from $L^{2}(\partial\mathbb{D})$) . We introduce a norm $\abs{\cdot}$ on $\mathbb{C}^{\abs{m}}\cong W/BW$ by
\[
\abs{x}:=\inf\{\norm{\sum_{j=1}^{\abs{m}}x_{j}e_{j}+Bg}{W}:\: g\in W\}.
\]
Let $\Norm{\cdot}{}$ be the matrix norm induced by $\abs{x}$. Lemma~\ref{schweinchenstinkt} yields 
\[
\Norm{\det(T_{\lambda})(T_{\lambda})^{-1}}{}=\phi_{\abs{m}}\left(\lambda_{1},\dots,\lambda_{\abs{m}}\right)
\]
The entries of $M_{S}$ with respect to $\{e_{j}\}_{j=1,...,\abs{m}}$ have been computed in~\cite[Proposition III.5]{SO1}. For $\lambda_{1}=\lambda_{2}=\dots=\lambda_{\abs{m}}=\lambda$ this matrix representation is exactly $T_\lambda$. To complete the proof it remains only to choose $\abs{m}=n$ and apply Lemma \ref{lem:Blaschke_power_lemma}.
\end{proof}

\subsection{Lower bounds for the resolvent}
We conclude this section studying Schäffer's question in a broader context applying the new method to bound the resolvent 
\[
R(\zeta,T)=(\zeta-T)^{-1}
\]
with given $\zeta\in\mathbb{D}$. To obtain a finite bound some kind of regularity regarding the location of eigenvalues of $T$ with respect to $\zeta$ must be assumed. In Schäffer's original discussion this regularization is achieved
through multiplication by the determinant. When $\zeta$ is shifted
away from the origin a natural generalization of Schäffer's question is to find the best $\cS_\zeta$ so that
\[
\abs{\det(b_{\zeta}(T))}\Norm{(\zeta-T)^{-1}}{}\leq\mathcal{S}_{\zeta}
\]
holds for any $T\in\cM_n$ with $\Norm{T}{}\leq1$ and $\sigma(T)\cap\zeta=\emptyset$. Applying Lemma~\ref{lem:Nikolai} to the polynomial $a(z)=\frac{m(\zeta)-m(z)}{\zeta-z}\frac{1}{m(\zeta)}$ and following the steps that led us to Theorem~\ref{thm:Main_theo} we find the more general result stated below.
\begin{thm} \label{lower bound} Let $T_{\lambda}\in\cM_n$ and $\Norm{\cdot}{}\in\cN(T_\lambda)$ be as in Theorem~\ref{thm:Main_theo}, let $\lambda\in(0,1)$ and $b_{\lambda}=\frac{z-\lambda}{1-\lambda z}$. Then we have that
\[
\abs{b_{\lambda}^{n}(\zeta)}\Norm{(\zeta-T_\lambda)^{-1}}{}\geq d\sqrt{n}.
\]
where $d=d(\lambda,\zeta)>0$ depends only on $\lambda$ and $\zeta$.
\end{thm} $\:$

\section{\label{Upper_bounds} Upper estimates related to Schäffer's question}

In this section we derive upper estimates for the resolvent of an
algebraic power-bounded operator. This is motivated by studying sharpness of the prefactor $\sqrt{e}$ in Schäffer's upper bound but also by continuing the existing line of research on this topic~\cite{DS,NN1}. Given $n\geq1$ and $1\leq C<\infty$
let $\cP_{n}(C)$ denote the set of power-bounded matrices/algebraic
operators $T$ with respect to any particular Banach norm $\Norm{\cdot}{}$,
$\sup_{k\geq0}\Norm{T^{k}}{}\leq C$. To ensure a finite bound $\zeta$
must be separated from the spectrum $\sigma(T)$ of $T$. It will
be convenient to measure this separation in \emph{Euclidean} distance
$
d(z,\, w)=\Abs{z-w}
$
or \emph{pseudo-hyperbolic} distance 
$
p(z,\, w)=\Abs{\frac{z-w}{1-\overline{z}w}}$
depending on the magnitude $\abs{\zeta}$. We write briefly $r\in(0,1)$
for the pseudo-hyperbolic distance between $\zeta$ and $\sigma(T)$.

\begin{thm} \label{main1} Let $T\in\cP_{n}(C)$ with minimal polynomial
$m=\prod_{i}^{\abs{m}}(z-\lambda_{i})$ of degree $\abs{m}$. The
following assertions hold: 
\begin{enumerate}
\item If $\abs{\lambda_{i}}=1$ for all $i=1,...,\abs{m}$ then for any
$\zeta\in\mathbb{C}\setminus\{\emph{zeros}(m)\}$ we have that 
\[
\Norm{R(\zeta,\, T)}{}\leq C\frac{\sqrt{\abs{m}}}{\min_{i}\abs{\zeta-\lambda_{i}}}.
\]

\item If $\zeta=0$ and $r=\min_{i}\abs{\lambda_{i}}>0$ we have that 
\[
\Norm{T^{-1}}{}\leq C\frac{\sqrt{\abs{m}(e-r^{2\abs{m}})}}{r^{\abs{m}}}.
\]

\item For $\zeta\in\mathbb{D}$ we have that 
\begin{align*}
\Norm{R(\zeta,\, T)}{}\leq Ce\sqrt{2}\frac{\sqrt{\abs{m}}}{\min_{i}\abs{1-\bar{\lambda}_{i}\zeta}r^{\abs{m}}}\sqrt{\frac{1}{1-r\Abs{\zeta}}+\frac{1}{2(1-r^{2})\abs{m}}},
\end{align*}
where $r=p(\zeta,\,\sigma)$. 
\item In case that $\zeta\in\partial\mathbb{D}$ we have 
\[
\Norm{R(\zeta,\, T)}{}\leq\frac{3}{2}C\sqrt{e^{2}-1}\frac{\abs{m}}{\min_{i}\abs{\zeta-\lambda_{i}}}.
\]

\end{enumerate}
\end{thm} 

Recall that for $\zeta=0$ Schäffer's original estimate reads $\Norm{T^{-1}}{}\leq\frac{\sqrt{en}}{r^{n}}$.
Theorem~\ref{Upper_bounds} point (2) is a slightly stronger bound.
To date all operators that served to provide lower bounds with regards
to Schäffer's conjecture had eigenvalues on a circle of radius $1-\frac{const}{n}$.
In this case the term $r^{n}$ does not go to zero and we conclude
that this class of examples has no hope to achieve Schäffer's upper
bound. More precisely we have the following corollary, which is a direct consequence of the proof of Theorem~\ref{main1}.
\begin{lem} \label{mainlemma} Let $T\in\cP_{n}(C)$ with minimal
polynomial $m$ of degree $\abs{m}$. Let $\lambda_{1},\,\lambda_{2},\,...,\,\lambda_{\abs{m}}$
denote the zeros of $m$. For any fixed $\zeta\in\mathbb{C}\setminus\{\emph{zeros}(m)\}$
and $\rho\in(0,1)$, we have that 
\begin{align*}
\Norm{R(\zeta,\, T)}{}{}^{2}\leq\frac{C}{1-\rho^{2}}\sum_{k=1}^{\abs{m}}\frac{1}{\rho^{2k-2}}\frac{1-\rho^{2}\left|\lambda_{k}\right|^{2}}{\left|\zeta-\lambda_{k}\right|^{2}}\prod_{j=1}^{k-1}\left|\frac{1}{b_{\lambda_{j}}(\zeta)}\left(1+\frac{(1-\rho^{2})\overline{\lambda_{j}}\zeta}{1-\overline{\lambda_{j}}\zeta}\right)\right|^{2}.
\end{align*}
\end{lem}
%
%
The proof of Theorem \ref{main1} is based on the following lemma. 
\begin{proof}[Proof of Lemma \ref{mainlemma}] Without loss
of generality we assume that $T$ can be diagonalized~\cite{NN1}. As $T\in\mathcal{P}_{n}(C)$
its spectrum $\sigma(T)=\left\{ \lambda_{1},\,\lambda_{2},\,...,\,\lambda_{\abs{m}}\right\} $
is contained in the closed unit disk $\overline{\mathbb{D}}$. We
can suppose that $\sigma(T)\subset\mathbb{D}$ as the general case
will follow by continuity. Let $H^2\subset L^2(\partial\mathbb{D})$ denote the standard Hardy space of the boundary $\partial{\mathbb{D}}$.
Given any function $f\in H^{2}$ and $\rho\in(0,1)$, we write $f_{\rho}(z):=f(\rho z)=\sum_{k\geq0}\hat{f}(k)\rho^{k}z^{k}$
and observe that by the Cauchy-Schwarz inequality and Plancherel's
identity 
\begin{align}
\Norm{f_{\rho}}{W}\leq\sqrt{\sum_{k\geq0}\abs{\hat{f}(k)}^{2}}\sqrt{\frac{1}{1-\rho^{2}}}=\Norm{f}{H^{2}}\sqrt{\frac{1}{1-\rho^{2}}}.\label{CS}
\end{align}
This inequality was used to obtain bounds on the inverse and resolvent
of a power-bounded operator in \cite{NN1,SO1} and to study spectral
convergence bounds for Markov chains in \cite{SRB}. From Remark~\ref{Rk_ext_rational_fct} above we have that
\begin{align*}
\Norm{R(\zeta,T)}{}\leq C\inf\left\{ \Norm{f}{W}\::\: f(\lambda_{j})=\frac{1}{\zeta-\lambda_{j}},\, j=1,\dots,\abs{m}\right\}.
\end{align*}
We fix $\rho\in(0,1)$ and consider the Blaschke product $\tilde{B}:=\prod_{i=1}^{n}\frac{z-\rho\lambda_{i}}{1-\rho\bar{\lambda}_{i}z}$,
whose zeros are contracted by a factor of $\rho$. The corresponding
Malmquist-Walsh basis for $K_{\tilde{B}}$ is 
\begin{align*}
\tilde{e}_{k}(z):=\frac{(1-\rho^{2}\abs{\lambda_{k}}^{2})^{1/2}}{1-\rho\bar{\lambda}_{k}z}\prod_{i=1}^{k-1}\frac{z-\rho\lambda_{i}}{1-\rho\bar{\lambda}_{i}z}.
\end{align*}
We write $P_{S}$ for the projector from $H^{2}$ to a subspace $S\subset H^{2}$.
Clearly any $f\in H^{2}$ can be decomposed as $f=P_{BH^{2}}f+P_{K_{B}}f$,
where we write $P_{K_{B}}=\sum_{k=1}^{n}\braket{\cdot}{e_{k}}e_{k}$
for the orthogonal projector onto $K_{B}$. Here $\braket{\cdot}{\cdot}$
means the scalar product on $L^{2}(\partial\mathbb{D})$, which is
consistent with the notation in Section~\ref{sec:A-constructive-method}.
Note that $(P_{BH^{2}}f)(\lambda_{i})=0$ such that 
\begin{align}
(P_{K_{B}}f)(\lambda_{i})=f(\lambda_{i})\quad\forall i=1,...,\abs{m}.\label{nodes}
\end{align}
Now Equation~\eqref{nodes} implies that 
\begin{align*}
P_{K_{\tilde{B}}}\left(\frac{1}{\zeta-z/\rho}\right)\Bigg|_{z=\rho\lambda_{i}}=\frac{1}{\zeta-\lambda_{i}}.
\end{align*}
On the other hand we have 
\begin{align*}
P_{K_{\tilde{B}}}\left(\frac{1}{\zeta-z/\rho}\right)=\sum_{k=1}^{\abs{m}}\braket{\tilde{e}_{k}}{\frac{1}{\zeta-z/\rho}}\tilde{e}_{k}(z)={\bar{\zeta}}^{-1}\sum_{k=1}^{\abs{m}}\tilde{e}_{k}\left(\frac{1}{\bar{\zeta}\rho}\right)\tilde{e}_{k}(z)
\end{align*}
and we conclude that 
\begin{align*}
{\bar{\zeta}}^{-1}\sum_{k=1}^{\abs{m}}\tilde{e}_{k}\left(\frac{1}{\bar{\zeta}\rho}\right)\tilde{e}_{k}(z)\Bigg|_{z=\rho\lambda_{i}}={\bar{\zeta}}^{-1}\sum_{k=1}^{\abs{m}}\tilde{e}_{k}\left(\frac{1}{\bar{\zeta}\rho}\right)\tilde{e}_{k}(\rho z)\Bigg|_{z=\lambda_{i}}=\frac{1}{\zeta-\lambda_{i}}.
\end{align*}
Together with the inequality in~\eqref{CS} this observation allows
us to bound the resolvent as 
\begin{align*}
\Norm{R(\zeta,T)}{} & \leq C\abs{\zeta}^{-1}\Norm{\sum_{k=1}^{\abs{m}}\tilde{e}_{k}\left(\frac{1}{\bar{\zeta}\rho}\right)\tilde{e}_{k}(\rho z)}{W}\leq C\abs{\zeta}^{-1}\sqrt{\frac{1}{1-\rho^{2}}}\Norm{\sum_{k=1}^{\abs{m}}\tilde{e}_{k}\left(\frac{1}{\bar{\zeta}\rho}\right)\tilde{e}_{k}(z)}{H_{2}}\\
 & =C\abs{\zeta}^{-1}\sqrt{\frac{1}{1-\rho^{2}}}\sqrt{\sum_{k=1}^{\abs{m}}\Abs{\tilde{e}_{k}\left(\frac{1}{\bar{\zeta}\rho}\right)}^{2}}.
\end{align*}
The last equality is exploiting orthonormality of Malmquist-Walsh
basis. A further straight forward computation shows that 
\[
\frac{1}{\vert\zeta\vert^{2}}\sum_{k=1}^{\abs{m}}\Abs{\tilde{e}_{k}\left(\frac{1}{\bar{\zeta}\rho}\right)}^{2}=\sum_{k=1}^{\abs{m}}\frac{1-\rho^{2}\left|\lambda_{k}\right|^{2}}{\left|\zeta-\lambda_{k}\right|^{2}}\frac{1}{\rho^{2(k-1)}}\prod_{j=1}^{k-1}\left|\frac{1-\bar{\lambda}_{j}\zeta}{\zeta-\lambda_{j}}\right|^{2}\prod_{j=1}^{k-1}\left|1+\frac{(1-\rho^{2})\overline{\lambda_{j}}\zeta}{1-\overline{\lambda_{j}}\zeta}\right|^{2},
\]
which proves Lemma~\ref{mainlemma}. \end{proof}

\begin{proof}{[}Proof of Theorem~\ref{main1}{]} Direct applications
of Lemma~\ref{mainlemma} prove the assertions of Theorem~\ref{main1}. 
\begin{enumerate}
\item In case that $\abs{\lambda_{i}}=1\ \forall i$ we have for any $\rho\in(0,1)$
that 
\begin{align*}
\Norm{R(\zeta,\, T)}{}^{2}\leq C^{2}\sum_{k=1}^{\abs{m}}\frac{1}{\rho{}^{2k-2}}\frac{1}{\left|\zeta-\lambda_{k}\right|^{2}}\prod_{j=1}^{k-1}\left|1+\frac{(1-\rho^{2})\overline{\lambda_{j}}\zeta}{1-\overline{\lambda_{j}}\zeta}\right|^{2}
\end{align*}
and taking the limit $\rho\uparrow1$ we find that 
$
\Norm{R(\zeta,\, T)}{}^{2}\leq C^{2}\min_{i}\frac{\abs{m}}{\left|\zeta-\lambda_{i}\right|^{2}}.
$

\item If $\zeta=0$ we have that 
\[
\Norm{T^{-1}}{}^{2}\leq C^{2}\inf_{0<\rho<1}\frac{1}{(1-\rho^{2})}\sum_{k=1}^{\abs{m}}\frac{1}{\rho^{2(k-1)}}\bigg(\frac{1}{\left|\lambda_{k}\right|^{2}}-\rho^{2}\bigg)\prod_{j=1}^{k-1}\frac{1}{\left|\lambda_{j}\right|^{2}}
\]
and summing the geometric series we find
\begin{align*}
\Norm{T^{-1}}{}^{2}\leq\inf_{0<\rho<1}\frac{\rho^{2}}{1-\rho^{2}}\left(\left(\frac{1}{\rho^{2}r^{2}}\right)^{\abs{m}}-1\right).
\end{align*}
Choosing $\rho^{2}=1-\frac{1}{\abs{m}+1}$ concludes the proof.

\item When $r=\min_{i}\Abs{\frac{\zeta-\lambda_{i}}{1-\bar{\lambda}_{i}\zeta}}>0$,
$\min_{i}\abs{1-\bar{\lambda}_{i}\zeta}>0$ and $\delta=\max_{i}\frac{1-\abs{\lambda_{i}}^{2}}{\abs{1-\bar{\lambda}_{i}\zeta}}$
we choose $\rho\in(0,1)$ with $1-\rho^{2}=\frac{\min_{i}\abs{1-\bar{\lambda}_{i}\zeta}}{2\abs{m}}$ and bound
\begin{align*}
 & \sum_{k=1}^{\abs{m}}\frac{1-\rho^{2}\left|\lambda_{k}\right|^{2}}{\left|\zeta-\lambda_{k}\right|^{2}}\frac{1}{\rho^{2(k-1)}}\prod_{j=1}^{k-1}\left|\frac{1-\bar{\lambda}_{j}\zeta}{\zeta-\lambda_{j}}\right|^{2}\prod_{j=1}^{k-1}\left|1+\frac{(1-\rho^{2})\overline{\lambda_{j}}\zeta}{1-\overline{\lambda_{j}}\zeta}\right|^{2}\\
 & \leq\frac{1}{r^{2}\min_{i}\abs{1-\bar{\lambda}_{i}\zeta}}\sum_{k=1}^{\abs{m}}\frac{1-\rho^{2}\left|\lambda_{k}\right|^{2}}{\min_{i}\abs{1-\bar{\lambda}_{i}\zeta}}\frac{1}{(r\rho)^{2(k-1)}}\left(1+\frac{1-\rho^{2}}{\min_{i}\abs{1-\bar{\lambda}_{i}\zeta}}\right)^{2(k-1)}\\
 & \leq\frac{\delta+\frac{1}{2\abs{m}}}{r^{2}\min_{i}\abs{1-\bar{\lambda}_{i}\zeta}}\frac{\left(\frac{1+\frac{1}{2\abs{m}}}{r\sqrt{1-\frac{1}{\abs{m}}}}\right)^{2\abs{m}}-1}{\left(\frac{1+\frac{1}{2\abs{m}}}{r\sqrt{1-\frac{1}{\abs{m}}}}\right)^{2}-1}\leq\frac{\delta+\frac{1}{2\abs{m}}}{r^{2\abs{m}}\min_{i}\abs{1-\bar{\lambda}_{i}\zeta}}\frac{e^{2}}{1-r^{2}}.
\end{align*}
The maximum of the function $\lambda\mapsto\frac{1-\Abs{\lambda}^{2}}{\abs{1-\bar{\lambda}\abs{\zeta}}}$
over the set $\left\{ \lambda\in\overline{\mathbb{D}}:\: p(\zeta,\,\lambda)\geq r\right\} )$ lies at $\lambda_{\textnormal{max}}=\frac{\Abs{\zeta}-r}{1-r\Abs{\zeta}}$
on the pseudo-hyperbolic circle of center $\Abs{\zeta}$ and radius $r$ with 
value $\delta_{{\rm max}}=\frac{1-\Abs{\lambda_{\max}}^{2}}{1-\bar{\lambda_{\max}}\Abs{\zeta}}=\frac{1-r^{2}}{1-r\Abs{\zeta}}$.
Note that for the optimization is it sufficient to consider real $\zeta$. Therefore
\[
\Norm{R(\zeta,\, T)}{}\leq\sqrt{2}eC\frac{\sqrt{\abs{m}}}{\min_{i}\abs{1-\bar{\lambda}_{i}\zeta}r^{\abs{m}}}\sqrt{\frac{1}{1-r\Abs{\zeta}}+\frac{1}{2(1-r^{2})\abs{m}}}.
\]

\item When $\abs{\zeta}=1$ then $r=1$. The reasoning is the same as in
(3). At fixed $s=\min_{i}\abs{1-\bar{\lambda}_{i}\zeta}\in\left[0,\,2\right]$
we notice that the sequence 
\[
\left(\frac{\left(\frac{1+\frac{1}{2\abs{m}}}{\sqrt{1-\frac{s}{2\abs{m}}}}\right)^{2\abs{m}}-1}{\left(\frac{1+\frac{1}{2\abs{m}}}{\sqrt{1-\frac{s}{2\abs{m}}}}\right)^{2}-1}\right)_{\abs{m}\geq2},
\]
is increasing and has limit
$\frac{2(e^{1+\frac{s}{2}}-1)}{s+2}$.
is increasing. We obtain that the resolvent is bounded by 
\[
\Norm{R(\zeta,\, T)}{}\leq2C\frac{\abs{m}}{\min_{i}\abs{\zeta-\lambda_{i}}}\sqrt{2+\frac{1}{2\abs{m}}}\sqrt{\frac{e^{1+\frac{s}{2}}-1}{s+2}}.
\]
\end{enumerate}
\end{proof}

\section{\label{sec:On-the-norm}Asymptotic analysis: On the $l^A_{\infty}-$norm of $(1-z^{2})b_{\lambda}^{n}$}
In this section we determine the asymptotic behavior of $\Norm{(1-z^{2})b_{\lambda}^{n}}{l_{\infty}^{A}}$, where $b=\frac{z-\lambda}{1-\bar{\lambda}z}$. Recall the contour integral representation of Fourier coefficients
\begin{align*}
\widehat{(1-z^{2})b_{\lambda}^{n}}(k)=\frac{1}{2i\pi}\oint_{\partial\mathbb{D}}(1-z^{2})b_{\lambda}^{n}(z)z^{-k}\frac{\d z}{z}.
\end{align*}
From this representation it is immediate that one can split
\begin{align*}
\widehat{(1-z^{2})b_{\lambda}^{n}}(k)=\widehat{b_{\lambda}^{n}}(k)-\widehat{b_{\lambda}^{n}}(k-2),\qquad k\geq2.
\end{align*}
In a preliminary work~\cite{SZ2} we developed the tools from asymptotic analysis and we determined the
asymptotic growth of the Taylor coefficients $\widehat{b_{\lambda}^{n}}(k)$
both with respect to $k$ and $n$. Holomorphy of $b_{\lambda}^{n}$ implies that
for any fixed $n$ the coefficients $\widehat{b_{\lambda}^{n}}(k)$
decay exponentially when $k$ grows large. In similar vein at fixed $k$ the coefficients $\widehat{b_{\lambda}^{n}}(k)$
decay exponentially in $n$. The interesting behavior, which is
relevant for the $l_{\infty}^{A}$-norm, therefore occurs when $k=k(n)$
is a sequence, see~\cite{SZ1} for details. 
We have shown~\cite[Proposition 2]{SZ1} that
\begin{enumerate}
\item if $\alpha\in(0,\alpha_{0})$ and $k\notin\left[\alpha n,\,\alpha^{-1}n\right]$
then $\widehat{b_{\lambda}^{n}}(k)$ decays exponentially in $n,$
\item if $\beta\in(\alpha_{0},1)$ and $k\in\left[\beta n,\,\beta^{-1}n\right]$
then $\widehat{b_{\lambda}^{n}}(k)=\cO\left(n^{-1/2}\right),$
\item If $k=\lfloor\alpha_{0}n\rfloor$ or $k=\lfloor\alpha_{0}^{-1}n\rfloor$
then $\widehat{b_{\lambda}^{n}}(k)=\cO\left(n^{-1/3}\right).$
\end{enumerate}

Bounding by triangular inequality we conclude that if $k\notin\left[\alpha n,\,\alpha^{-1}n\right]$ then $\widehat{(1-z^{2})b_{\lambda}^{n}}(k)$ decays exponentially in $n$. Moreover for any sequence $k=k(n)$ we find that $\widehat{(1-z^{2})b_{\lambda}^{n}}(k)$ decays at most as $\cO(n^{-1/3})$. Building on the methods developed for $b_{\lambda}^{n}(k)$ the goal of this section is to determine more precisely the asymptotics of $\widehat{(1-z^{2})b_{\lambda}^{n}}(k)$. Our findings are contained in below proposition and summarized in Table~\ref{WideOpen} below.
\begin{figure}
\label{WideOpen}
\begin{tabular}{|c|c|c|}
\hline
Values of $k(n)$ in interval  & Decay of $\widehat{(1-z^{2})b_{\lambda}^{n}}(k)$  & Region\tabularnewline
\hline
\hline
$[0,\,\alpha n]$  & exponential  & I\tabularnewline
\hline
$(\alpha n,\,\alpha_{0}n-n^{1/3}]$  & $\frac{\left(\alpha_{0}-k/n\right)^{1/4}}{n^{1/2}}\exp\left(-\frac{2}{3}n\left(\alpha_{0}-k/n\right)^{3/2}\right)$  & II\tabularnewline
\hline
$[\alpha_{0}n-n^{1/3},\,\alpha_{0}n+n^{1/3}]$  & $\frac{1}{n^{2/3}}$  & III\tabularnewline
\hline
$[\alpha_{0}n+n^{1/3},\,\alpha_{0}^{-1}n-n^{1/3}]$  & $\frac{\left(\frac{k}{n}-\alpha_{0}\right)^{1/4}\left(\alpha_{0}^{-1}-\frac{k}{n}\right)^{1/4}}{n^{1/2}}$  & IV\tabularnewline
\hline
$[\alpha_{0}^{-1}n-n^{1/3},\,\alpha_{0}^{-1}n+n^{1/3}]$  & $\frac{1}{n^{2/3}}$  & V\tabularnewline
\hline
$[\alpha_{0}^{-1}n+n^{1/3},\,\alpha^{-1}n)$  & $\frac{\left(\frac{k}{n}-\alpha_{0}^{-1}\right)^{1/4}}{n^{1/2}}\exp\left(-\frac{2}{3}n\left(\frac{k}{n}-\alpha_{0}^{-1}\right)^{3/2}\right)$  & VI\tabularnewline
\hline
$[\alpha^{-1}n,\,\infty)$  & exponential  & VII\tabularnewline
\hline
\end{tabular}
\caption{Illustration of  asymptotic bounds for $\abs{\widehat{(1-z^{2})b_{\lambda}^{n}}(k)}$ as a function of $k=k(n)$ (up to a multiplicative constant). Here $\alpha_{0}:=\frac{1-\lambda}{1+\lambda}$
and $\alpha\in(0,\alpha_{0})$ arbitrary but fix.}
\end{figure}
\begin{prop} \label{thm:main} Let $\lambda\in(0,1),$ $b_{\lambda}=\frac{z-\lambda}{1-\lambda z}$
and $n\geq1$. Set $\alpha_{0}:=\frac{1-\lambda}{1+\lambda}$ and
choose fixed $\alpha\in(0,\alpha_{0})$ and $\beta\in(\alpha_{0},1)$. In the following we consider
sequences $k=k(n)$ and all assertions are meant to hold for large
enough $n$.
\begin{enumerate}
\item If $k/n\leq\alpha$ then $\abs{\widehat{(1-z^{2})b_{\lambda}^{n}}(k)}$
decays exponentially as $n$ tends to $\infty$. Similarly if $k/n\geq\alpha^{-1}$
then $\abs{\widehat{(1-z^{2})b_{\lambda}^{n}}(k)}$ decays exponentially
as $n$ tends to $\infty$.
\item If $k/n\in(\alpha,\alpha_{0}-n^{-2/3})\cup(\alpha_{0}^{-1}+n^{-2/3},\alpha^{-1})$
then we have the following asymptotic growth estimate
\begin{align*}
 & \Abs{\widehat{(1-z^{2})b_{\lambda}^{n}}(k)}\lesssim\frac{(\min\left((\alpha_{0}-k/n),\,(k/n-\alpha_{0}^{-1})\right))^{1/4}}{n^{1/2}}\:*\\
 & \exp\left(-\frac{2}{3}n(\min\left((\alpha_{0}-k/n),\,(k/n-\alpha_{0}^{-1})\right))^{3/2}\right).
\end{align*}

\item If $k/n\in[\alpha_{0}-n^{-2/3},\alpha_{0}+n^{-2/3})\cup(\alpha_{0}^{-1}-n^{-2/3},\alpha_{0}^{-1}+n^{-2/3}]$
then
\[
\abs{\widehat{(1-z^{2})b_{\lambda}^{n}}(k)}\lesssim\frac{1}{n^{2/3}}.
\]

\item If $k/n\in(\alpha_{0}+n^{-2/3},\beta]\cup[\beta^{-1},\alpha_{0}^{-1}-n^{-2/3})$
then
\[
\Abs{\widehat{(1-z^{2})b_{\lambda}^{n}}(k)}\lesssim\frac{\left(\frac{k}{n}-\alpha_{0}\right)^{1/4}\left(\alpha_{0}^{-1}-\frac{k}{n}\right)^{1/4}}{n^{1/2}}.
\]

\item If $k/n\in(\beta,\beta^{-1})$
then
\begin{align*}
 &\Abs{\widehat{(1-z^{2})b_{\lambda}^{n}}(k)}\leq\sqrt{\frac{2(1-\lambda^{2})}{\pi n}}\frac{(k/n-\alpha_{0})^{1/4}(\alpha_{0}^{-1}-k/n)^{1/4}}{\lambda(k/n)^{3/2}}\big(1+O(n^{-1})\big).
\end{align*}
\end{enumerate}
\end{prop}
To determine the upper bounds  for $k\in\left[\alpha n,\,\alpha^{-1}n\right]$ we rely on methods from the asymptotic analysis of Fourier integrals. We introduce a function $f_{a}$ with $a\in\mathbb{R}^{+}$ by
\begin{align*}
f_{a}(z) &:=\log{\left(\frac{z^{a}(1-\lambda z)}{z-\lambda}\right)},
\end{align*}
where $\log$ denotes the principal branch of the complex logarithm.
Using $f_a$ we can write
\begin{align}\label{integral}
\overline{\widehat{(1-z^{2})b_{\lambda}^{n}}(k)}=\frac{1}{2i\pi}\oint_{\partial\mathbb{D}}g(z)\exp\left(nf_{a}(z)\right)\frac{\d z}{z},
\end{align}
where $g(z)=(1-z^{-2})$. Determining the asymptotic behavior of such integrals as $n\rightarrow+\infty$ is a relatively standard task when $f$ is fixed, see e.g.~\cite{WR,BlHa}. For us the situation is slightly more complicated as $f_a$ depends on $k$ and $n$ but even here we can rely on existing methodology. It is common that the dominant contribution to such integrals comes from a small neighborhood around the stationary points
of $f$. We begin by identifying those points for $f_{a}=f_{k/n}$. For notational convenience we shall write $f_{a}$
with an additional argument instead of the index $f_{a}(z)=f(z,a).$
\begin{lem}\label{lem:critical_pts_f} Let $f_a(z)=f(a,z)$ be as defined above and let $a=\frac{k}{n}$. We have the following assertions.
\begin{enumerate}
\item If $a\in(\alpha_{0},\,\alpha_{0}^{-1})$ then $f(\cdot,a)$ has two distinct stationary points $z_{\pm}\in\partial\mathbb{D}$
of order one, i.e.~$\frac{\partial f}{\partial z}\left(z_{\pm},a\right)=0$
but $\frac{\partial^{2}f}{\partial z^{2}}\left(z_{\pm},a\right)\neq0$,
satisfying $z_{-}=\overline{z_{+}}$.

\item If $a\in\left\{\alpha_{0},\alpha_{0}^{-1}\right\}$
then $f(\cdot,a)$ has one stationary point $z_{0}\in\left\{-1,1\right\}$ of order two i.e,~$\frac{\partial f}{\partial z}\left(z_{0},a\right)=\frac{\partial^{2}f}{\partial z^{2}}\left(z_{0},a\right)=0$
but $\frac{\partial^{3}f}{\partial z^{3}}\left(z_{0},a\right)\neq0$.

\item If $a\notin[\alpha_{0},\,\alpha_{0}^{-1}]$ then $f(\cdot,a)$ has two stationary points $z_{\pm}\in\mathbb{R}$
of order one, i.e. $\frac{\partial f}{\partial z}\left(z_{\pm},a\right)=0$
but $\frac{\partial^{2}f}{\partial z^{2}}\left(z_{\pm},a\right)\neq0$, satisfying $z_{-}=z_{+}^{-1}$.
\end{enumerate}
The stationary points $z_{+}$ and $z_{-}$ are given explicitly by
\begin{equation}
z_{\pm}=\frac{a(1+\lambda^{2})-(1-\lambda^{2})}{2\lambda a}\pm\sqrt{\left(\frac{a(1+\lambda^{2})-(1-\lambda^{2})}{2\lambda a}\right)^{2}-1}\label{eq:formula_z_plus_minus}.
\end{equation}
%
%

%
%
\end{lem}
\begin{proof}
Computing derivatives we confirm
\begin{align*}
\frac{\partial f}{\partial z} & =-\frac{1}{z-\lambda}+\frac{a}{z}-\frac{\lambda}{1-\lambda z},\\
\frac{\partial^{2}f}{\partial z^{2}} & =\frac{1}{(z-\lambda)^{2}}-\frac{a}{z^{2}}-\frac{\lambda^{2}}{(1-\lambda z)^{2}},\\
\frac{\partial^{3}f}{\partial z^{3}} & =-\frac{2}{(z-\lambda)^{3}}+\frac{2a}{z^{3}}-\frac{2\lambda^{3}}{(1-\lambda z)^{3}}.
\end{align*}
The function $f(z,a)$ has a stationary point if and only if $\partial f/\partial z=0$,
i.e.~iff
\begin{align*}
a=1+\frac{\lambda}{z-\lambda}+\frac{\lambda z}{1-\lambda z}.
\end{align*}
Solving the latter for $z$ yields the representation \eqref{eq:formula_z_plus_minus} for the roots $z_\pm$ of $\frac{\partial f}{\partial z}$. If $a\notin\left\{ \alpha_{0},\,\alpha_{0}^{-1}\right\} $ then $z_{+}$ and $z_{-}$ are distinct. If $a\in(\alpha_{0},\,\alpha_{0}^{-1})$ then $z_\pm\in\partial\mathbb{D}-\{-1,1\}$ and if $a\notin[\alpha_{0},\,\alpha_{0}^{-1}]$ then $z_\pm\in\mathbb{R}-\{-1,1\}$.
Plugging in we see that
\begin{eqnarray}
\frac{\partial^{2}f}{\partial z^{2}}\Bigg|_{z=z_{\pm}} & = & \frac{(1-\lambda^{2})(1-z_{\pm}^{2})\lambda}{z_{\pm}(z_{\pm}-\lambda)^{2}(1-\lambda z_{\pm})^{2}}.\label{eq:scd_deriv}
\end{eqnarray}
If $a\in\left\{ \alpha_{0},\,\alpha_{0}^{-1}\right\} $ then $\frac{\partial f}{\partial z}$
has a unique zero. If $a=\alpha_{0}^{-1}$ then $z_{+}=z_{-}=1=\tilde{z}_{0}$
and
\[
f(1,\alpha_{0}^{-1})=\frac{\partial f}{\partial z}(1,\alpha_{0}^{-1})=\frac{\partial^{2}f}{\partial z^{2}}(1,\alpha_{0}^{-1})=0,
\]
with
\[
\frac{\partial^{3}f}{\partial z^{3}}(1,\alpha_{0}^{-1})=-\frac{2\lambda(1+\lambda)}{(1-\lambda)^{3}}\neq0.
\]
If $a=\alpha_{0}$ then $z_{+}=z_{-}=-1=z_{0}$ and
\[
\frac{\partial f}{\partial z}(-1,\alpha_{0})=\frac{\partial^{2}f}{\partial z^{2}}(-1,\alpha_{0})=0,\qquad\frac{\partial^{3}f}{\partial z^{3}}(-1,\alpha_{0})=-\frac{2\lambda(1-\lambda)}{(1+\lambda)^{3}}\neq0.
\]
\end{proof}
The lemma shows that the location of stationary points of $f_a$ in $\mathbb{C}$ is determined by the location of $a=\frac{k}{n}$ relative to the critical interval $[\alpha_0,\alpha_0^{-1}]$. As $a$ approaches the boundary the stationary points degenerate. Thus we treat the situations, where $a$ is separated from the boundary and where $a$ approaches the boundary individually. The former scenario corresponds to point~\emph{5)} in Proposition~\ref{thm:main}, i.e.~there is a $\beta\in(\alpha_0,1)$ that separates $a$ from the boundary, $a\in(\beta,\beta^{-1})$. In the second scenario, where $a$ approaches the boundary, the asymptotic behavior depends on the speed at which $a$ approaches the boundary. This is reflected in the points \emph{2), 3), 4)} from Proposition~\ref{thm:main}. Speaking roughly we employ the following methods to determine the asymptotics.
\begin{itemize}
\item If $a$ is separated from the boundary then the stationary points $z_{\pm}$ of $f_a$ belong
to the contour of integration $\partial\mathbb{D}$. Since $\Abs{z^{k/n}\frac{1-\lambda z}{z-\lambda}}=1$
for any $z\in\partial\mathbb{D}$ we can introduce the real
function
\[
h(\varphi)=h_{a}(\varphi):=-if(e^{i\varphi},a)\qquad\varphi\in[0,\pi],
\]
to write the integral as a generalized Fourier integral,
\begin{align*}
\overline{\widehat{(1-z^{2})b_{\lambda}^{n}}(k)}& =\frac{1}{2\pi}\int_{-\pi}^{\pi}g(z)e^{nf_a(z)}\Big|_{z=e^{i\varphi}}\d\varphi\nonumber 
=\frac{1}{\pi}\Re{\left\{ \int_{0}^{\pi}g(z)e^{inh_a(z)}\Big|_{z=e^{i\varphi}}\d\varphi\right\} }.
\end{align*}
Since $a$ is separated from the critical boundary we have
\begin{equation*}
\min_{a\in[\beta,\beta^{-1}]}\Abs{\frac{\partial^{2}h_{a}}{\partial \varphi^{2}}(\varphi_{a},a)}>0,
\end{equation*}
where $\varphi_{a}$ is the unique critical point of $h_{a}$ in $(0,\pi)$.
To determine the asymptotic behavior of this integral we will rely on the \textit{method of stationary phase}~\cite{EA}.
\item When $a$ approaches the critical boundary we are faced with coalescing saddle points. If $a$ approaches the boundary from the inside the two saddle points $z_\pm$ remain on $\partial\mathbb{D}$. However, when $a$ approaches the boundary from the outside the saddle points $z_\pm$ move along the real line. While in the former situation we can rely on a modified version of the method of stationary phase, in the latter case we first deform the contour of integration such that is passes through $z_\pm$.
The asymptotic behavior in this situation is determined using the~\emph{method of steepest descents}. To capture the asymptotic behavior we also take account of the speed at which $a$ approaches the boundary. To achieve this we employ a~\emph{uniform} version of the method of stationary phase/ steepest descents as is introduced in~\cite{CFU}.
\end{itemize}

\begin{proof}[Proof of Proposition 9, Point 5)]
To determine an asymptotic upper estimate we suppose a situation where $a$ is \emph{fixed} in an interval $K\subset[\beta,\beta^{-1}]$. The stationary points of $h=h_{a}$ are given by (see Lemma \ref{lem:critical_pts_f} point 1))
\[
z_{+,-}=\frac{a(1+\lambda^{2})-(1-\lambda^{2})}{2\lambda a}\pm i\sqrt{1-\left(\frac{a(1+\lambda^{2})-(1-\lambda^{2})}{2\lambda a}\right)^{2}}\in\partial\mathbb{D}
\]
and we write $z_{+,-}=e^{i\varphi_{+,-}}$ with $\varphi_{+}\in[0,\pi]$
and $\varphi_{-}\in(-\pi,0]$. Only $z_{+}$ is interesting because
we integrate over $[0,\pi]$. For the second derivative we have that
\[
ih''(\varphi)=\frac{\partial}{\partial\varphi}\left(\frac{\partial f}{\partial z}\frac{dz}{d\varphi}\right)=\frac{\partial^{2}f}{\partial z^{2}}\left(\frac{dz}{d\varphi}\right)^{2}+\frac{\partial f}{\partial z}\frac{d^{2}z}{(d\varphi)^{2}}.
\]
It follows from \eqref{eq:scd_deriv} that
\begin{eqnarray*}
i\frac{\partial^{2}h}{\partial\varphi^{2}}\Bigg|_{\varphi=\varphi_{+}} & = &-z_{+}^{2} \frac{(1-\lambda^{2})(1-z_{+}^{2})\lambda}{z_{+}(z_{+}-\lambda)^{2}(1-\lambda z_{+})^{2}}
\end{eqnarray*}
so that $h''(\varphi_{+})>0$. To find the asymptotics we apply a
standard result by A. Erdélyi~\cite[Theorem 4]{EA} (see also F.
Olver \cite[Theorem 1]{OF} for a more explicit form), which however
requires that the stationary point is an endpoint of the interval
of integration. Hence we begin by splitting
\[
\int_{0}^{\pi}g(\varphi)e^{inh(\varphi)}\d\varphi=\int_{0}^{\varphi_{+}}g(\varphi)e^{inh(\varphi)}\d\varphi+\int_{\varphi_{+}}^{\pi}g(\varphi)e^{inh(\varphi)}\d\varphi.
\]
For the second integral \cite[Theorem 4]{EA} (see also \cite[Theorem 1]{OF})
gives
\begin{align*}
 \int_{\varphi_{+}}^{\pi}g(\varphi)e^{inh(\varphi)}\d\varphi
 &=\frac{1}{2}\Gamma(1/2)k(0)e^{i\frac{\pi}{4}}n^{-1/2}e^{inh(\varphi_{+})}
 +\frac{1}{2}\Gamma(1)k'(0)e^{i\frac{\pi}{2}}n^{-1}e^{inh(\varphi_{+})}\\
 & -\frac{i}{n}e^{inh(\pi)}\frac{g(\pi)}{h'(\pi)}+\cO\left(n^{3/2}\right),
\end{align*}
with
\[
k(0)=2^{1/2}g(\varphi_{+})\left(h''(\varphi_{+})\right)^{-1/2}
\]
and
\[
k'(0)=\frac{2}{h''(\varphi_{+})}g'(\varphi_{+})-\frac{2}{h''(\varphi_{+})}\frac{h^{(3)}(\varphi_{+})}{3h''(\varphi_{+})}g(\varphi_{+}).
\]
For the first integral $\int_{0}^{\varphi_{+}}g(\varphi)e^{inh(\varphi)}\d\varphi$
we change the variable of integration $\varphi\mapsto-\varphi$ as
suggested in \cite[p.~23]{EA}. We get
\[
\int_{0}^{\varphi_{+}}g(\varphi)e^{inh(\varphi)}\d\varphi=\int_{-\varphi_{+}}^{0}g(-\varphi)e^{inh(-\varphi)}\d\varphi.
\]
Applying \cite[Theorem 4]{EA} (see also \cite[Theorem 1]{OF}) gives
\begin{align*}
 & \int_{-\varphi_{+}}^{0}g(-\varphi)e^{inh(-\varphi)}\d\varphi
 =\frac{1}{2}\Gamma(1/2)k(0)e^{i\frac{\pi}{4}}n^{-1/2}e^{inh(\varphi_{+})}
 +\frac{1}{2}\Gamma(1)k'(0)e^{i\frac{\pi}{2}}n^{-1}e^{inh(\varphi_{+})}\\
 & -\frac{i}{n}e^{inh(0)}\frac{g(0)}{h'(0)}+\cO\left(n^{3/2}\right)
\end{align*}
with $k(0)=2^{1/2}g(\varphi_{+})\left(h''(\varphi_{+})\right)^{-1/2}$ and
$
k'(0)=-\frac{2}{h''(\varphi_{+})}g'(\varphi_{+})+\frac{2}{h''(\varphi_{+})}\frac{h^{(3)}(\varphi_{+})}{3h''(\varphi_{+})}g(\varphi_{+}).
$
Observing that $h(0)=0$ while $h(\pi)=(a-1)\pi,$ $g(0)=0$ while
$g(\pi)=0,$ and $h'(0)=\frac{(a-1)(1-\lambda)-2\lambda}{1-\lambda}$
while $h'(\pi)=-\frac{(a-1)(1+\lambda)+2\lambda}{1+\lambda}$ we get
$
-\frac{i}{n}e^{inh(\pi)}\frac{g(\pi)}{h'(\pi)}=0
$
while
$
-\frac{i}{n}e^{inh(0)}\frac{g(0)}{h'(0)}=0.
$
We conclude that
\begin{align*}
 \int_{0}^{\pi}g(\varphi)e^{inh(\varphi)}\d\varphi
 &=\Gamma(1/2)\left(2^{1/2}g(\varphi_{+})\left(h''(\varphi_{+})\right)^{-1/2}\right)e^{i\frac{\pi}{4}}n^{-1/2}e^{inh(\varphi_{+})}+\cO\left(n^{-3/2}\right)\\
 & =e^{inh(\varphi_{+})+i\frac{\pi}{4}}(1-z_{-}^{2})\left(\frac{2\abs{z_{+}-\lambda}^{4}}{n\lambda(1-\lambda^{2})\abs{1-z_{+}^{2}}}\right)^{1/2}\Gamma(1/2)+\cO\left(n^{-3/2}\right)\\
 & =e^{inh(\varphi_{+})-i\varphi_{+}+i\frac{3\pi}{4}}\frac{(a-\alpha_{0})^{1/4}(\alpha_{0}^{-1}-a)^{1/4}}{\lambda a^{3/2}}\left(\frac{2\pi(1-\lambda^{2})}{n}\right)^{1/2}+\cO\left(n^{-3/2}\right),
\end{align*}
where we made use of
$
\abs{z_{+}-\lambda}=\sqrt{\frac{1-\lambda^{2}}{a}}
$
and
$
2\Im(z_{+})=\abs{z_{+}^{2}-1}=\frac{\sqrt{(a-\alpha_{0})(\alpha_{0}^{-1}-a)}}{\lambda a}.
$
We conclude that for fixed $a$ we have
\begin{align*}
&\frac{1}{\pi}\Re\left\{\int_{0}^{\pi}g(\varphi)e^{inh(\varphi)}\d\varphi\right\}\\
& =\sqrt{\frac{2(1-\lambda^{2})}{\pi n}}\frac{(a-\alpha_{0})^{1/4}(\alpha_{0}^{-1}-a)^{1/4}}{\lambda a^{3/2}}\cos\left(nh(\varphi_{+})-\varphi_{+}+\frac{3\pi}{4}\right)\big(1+O(n^{-1})\big).
\end{align*}
We are interested in an asymptotic bound to $\sup_{k\in[\beta n,\beta^{-1}n]}\Abs{\widehat{(1-z^{2})b_{\lambda}^{n}}(k)}$. If $n$ is large we see from the above formula that the supremum is attained at $k=a^*n$.
\end{proof}
The situation is more complicated when $k$ approaches the boundary of $[\alpha_{0}n,\,\alpha_{0}^{-1}n]$. When $a$ varies in a domain of the complex plane the saddle points $z_\pm$ vary with $a$ and coalesce when a approaches the critical boundary. If $a$ was fixed the method of stationary phase / steepest descents would apply but if $a$ approaches the critical boundary the radius of convergence of the resulting asymptotic expansions goes to $0$. The so-called \emph{uniform method of steepest decent}~\cite{CFU} was developed to provide asymptotic expansions that are uniform in $a$. According to Lemma \ref{lem:critical_pts_f} the integral
\begin{align*}
\frac{1}{2i\pi}\oint_{\partial\mathbb{D}}g(z)\exp\left(nf_{a}(z)\right)\frac{\d z}{z},
\end{align*}
has two simple saddle points when $k$ is separated from the boundary. The saddle points coalesce into a single saddle point of order two as $a$ approaches the boundary. We assume that $k$ approaches the right boundary, $\lim_{n\rightarrow\infty}\frac{k}{n}=\alpha_{0}^{-1}$,
the reasoning for the left one being similar. In this situation the main contribution to the contour integral comes from a small arc around $z=1$. We deform the circle $\partial\mathbb{D}$ such that the new contour of integration $\cC$ passes through the saddle points $z_\pm$. We write $\cD_\pm$ for a small neighborhood containing $z=1$ and $z_\pm$. We have
\begin{align*}
\frac{1}{2i\pi}\oint_{\partial\mathbb{D}}g(z)\exp\left(nf_{a}(z)\right)\frac{\d z}{z}=\frac{1}{2i\pi}\int_{\cC\cap\cD_\pm}g(z)\exp\left(nf_{a}(z)\right)\frac{\d z}{z}+\cO\left(\frac{1}{n}\right).
\end{align*}
This is a consequence of \cite[Lemma 7]{SZ1} where it is shown that for any fixed $\epsilon\in(0,\pi)$
\[
\int_{\partial\mathbb{D}\setminus\mathcal{C}_{\varepsilon}}\exp\left(nf_{a}(z)\right)\frac{\d z}{z}=\cO\left(\frac{1}{n}\right)
\]
where $\cC_\varepsilon=\{z=e^{i\varphi}\:|\:\varphi\in(-\varepsilon,\varepsilon)\}$.
To simplify the dependence of the saddle points on $a$ we change the variable of integration via a locally one-to-one transformation of the form
\[
f(z,a)=-\frac{t^{3}}{3}+\gamma^{2}t+\rho.
\]
The parameters $\gamma$ and $\rho$ are determined such that $t=0$ is mapped to $z=1$ and the saddle points $z_\pm$ are mapped symmetrically to $t=\pm\gamma$. The below proposition is from~\cite{CFU} stated in the formulation of~\cite[Theorem 1 p.~368]{WR}.

\begin{prop}[\cite{CFU}] \label{cubic} For $a$ near $\alpha_{0}^{-1}$
the cubic transformation
\[
f(z,a)=-\frac{t^{3}}{3}+\gamma^{2}t
\]
with $$\gamma^{2}=\frac{(a-\alpha_{0}^{-1})(1-\lambda)}{\left(\lambda(1+\lambda)\right)^{1/3}}+o((a-\alpha_{0}^{-1}))$$
has exactly one branch $t=t(z,a)$ that can be expanded into a power
series in $z$ with coefficients that are continuous in $a$. On
this branch the points $z=z_{\pm}$ correspond to $t=\pm\gamma$.
The mapping of $z$ to $t$ is one-to-one on a small neighborhood $\cD_{\pm}$ of $z=1$ containing $z_{+}$ and $z_{-}$. \end{prop}
This is proved in \cite[Appendix]{SZ1} for $a$ near $\alpha_{0}^{-1}$ and $a<\alpha_{0}^{-1}$. We omit the corresponding proof for  $a>\alpha_{0}^{-1}$, which is similar. With the results of~\cite{CFU} we obtain a uniform expansion of the integral in terms of the Airy function $Ai$. For real arguments the latter can be defined as an improper Riemann
integral
\[
Ai(x)=\frac{1}{\pi}\int_{0}^{\infty}\cos\left(\frac{t^{3}}{3}+xt\right)\d t.
\]
For large negative arguments the Airy function shows oscillatory behavior
\[
Ai(-x)\sim\frac{1}{x^{1/4}\sqrt{\pi}}\cos\left(\frac{2}{3}x^{3/2}-\frac{\pi}{4}\right),\ Ai'(-x)\sim\frac{x^{1/4}}{\sqrt{\pi}}\sin\left(\frac{2}{3}x^{3/2}-\frac{\pi}{4}\right),\  x\rightarrow+\infty,
\]
and exponential behavior for large positive arguments
\[
Ai(x)\sim\frac{1}{2x^{1/4}\sqrt{\pi}}\exp\left(-\frac{2}{3}x^{3/2}\right),\ 
Ai'(x)\sim-\frac{x^{1/4}}{2\sqrt{\pi}}\exp\left(-\frac{2}{3}x^{3/2}\right),\ x\rightarrow+\infty.
\]

To determine the asymptotics of the integral in a first step we apply the transformation of Proposition \ref{cubic}. We adopt the notation $\hat{\cD}_{\pm}$ from \cite{BlHa} for the image
of $\cD_{\pm}$ under the transformation $z\mapsto t(z)$. Similarly $\hat\cC$ is the image of $\cC$. We have
\[
\frac{1}{2i\pi}\int_{\cC\cap\cD_\pm}g(z)\exp\left(nf_{a}(z)\right)\frac{\d z}{z}=\frac{1}{2i\pi}\int_{\hat\cC\cap\hat{D}_{\pm}}G_{0}(t)\exp\left(n\left(-\frac{t^{3}}{3}+\gamma^{2}t\right)\right)dt
\]
with
\[
G_{0}(t)=\psi(z(t))\frac{dz}{dt},\qquad\psi(z)=\frac{1-z^{-2}}{z},
\]
which is regular in $\hat{\cD}_{\pm}$. We exploit the fact that when the
integrand vanishes near a critical point, the contribution to asymptotics
from that point is diminished. Thus we expand
\[
G_{0}(t)=A_{0}+A_{1}t+(t^{2}-\gamma^{2})H_{0}(t).
\]
As long as $H_{0}$ is regular in $\hat{\cD}_{\pm}$ the last term of the
above identity vanishes at the two saddle points $t=\pm\gamma$. We
can then determine $A_{0},$ $A_{1}$ by setting $t=\pm\gamma$ in the above to get
\[
A_{0}=\frac{G_{0}(\gamma)+G_{0}(-\gamma)}{2},\:A_{1}=\frac{G_{0}(\gamma)-G_{0}(-\gamma)}{2\gamma}.
\]
With $A_{0},$ $A_{1}$ so determined it is shown in \cite{BlHa}
that $H_{0}=\frac{G_{0}(t)-A_{0}-A_{1}t}{t^{2}-\gamma^{2}}$ is regular in $\hat{\cD}_{\pm}$ as desired. We conclude
\begin{equation*}
\frac{1}{2i\pi}\int_{\cC\cap\cD_\pm}g(z)\exp\left(nf_{a}(z)\right)\frac{\d z}{z}\sim\frac{1}{2i\pi}\int_{\hat\cC\cap\hat{\cD}_{\pm}}(A_{0}+A_{1}t)\exp\left(n\left(-\frac{t^{3}}{3}+\gamma^{2}t\right)\right)dt+R_{0}(n)
\end{equation*}
with $R_{0}(n)=\int_{\hat\cC\cap\hat{\cD}_{\pm}}(t^{2}-\gamma^{2})H_{0}(t)\exp\left(n\left(-\frac{t^{3}}{3}+\gamma^{2}t\right)\right)dt$.
Following the arguments of \cite[371-373]{BlHa}
we rewrite the integral replacing $\hat{\cC}\cap\hat{\cD}_{\pm}$ with the asymptotically equivalent $C_{1}\cap\hat{\cD}_{\pm}$,
where $C_{1}$ starts at infinity with argument $-2\pi/3$
and ends at infinity with argument $2\pi/3$.
The introduced error is negligible, asymptotically vanishing faster than the integral.
With the Airy function we write
\[
Ai(x)=\frac{1}{2i\pi}\int_{C_{1}}\exp\left(\left(-\frac{s^{3}}{3}+sx\right)\right)ds=\frac{2}{\pi}\int_{-\infty}^{+\infty}\cos\left(\frac{\tau^{3}}{3}+\tau x\right)d\tau
\]
in order to express the integral in terms of $Ai(x)$ and its derivative
$Ai'(x)$. To determine $R_{0}$ we integrate by parts and introduce
another asymptotically negligible error by ignoring the boundary terms.
We find
\[
\frac{1}{2i\pi}\int_{\cC\cap\cD_\pm}g(z)\exp\left(nf_{a}(z)\right)\frac{\d z}{z}\sim\frac{A_{0}}{n^{1/3}}Ai(n^{2/3}\gamma^{2})+\frac{A_{1}}{n^{2/3}}Ai'(n^{2/3}\gamma^{2})+R_{1}(n)
\]
with
\[
R_{1}(n)=\frac{1}{n}\int_{\hat\cC\cap\hat{\cD}_{\pm}}G_{1}(t)\exp\left(n\left(-\frac{t^{3}}{3}+\gamma^{2}t\right)\right)dt,\qquad G_{1}(t)=H_{0}'(t).
\]
Repeating the same procedure for $\int_{C_{1}\cap\hat{D}_{\pm}}G_{1}(t)\exp\left(n\left(-\frac{t^{3}}{3}+\gamma^{2}t\right)\right)dt$, i.e.~with $G_{0}$ replaced by $G_{1}$ we find that $R_{1}(n)$ can be neglected
\[
\frac{1}{2i\pi}\oint_{\partial\mathbb{D}}g(z)\exp\left(nf_{a}(z)\right)\frac{\d z}{z}\sim\frac{A_{0}}{n^{1/3}}Ai(n^{2/3}\gamma^{2})+\frac{A_{1}}{n^{3/3}}Ai'(n^{2/3}\gamma^{2}),\qquad n\rightarrow\infty.
\]
To estimate $A_{0}$ and $A_{1}$ it remains to explicit
\[
G_{0}(\pm\gamma)=G_{0}(t_{\pm})=z_{\pm}^{-2}(z_{\mp}-z_{\pm})z'(t_{\pm}).
\]
Observe that as $k/n$ approaches $\alpha_{0}^{-1}$ we have $\gamma=\cO\left(\sqrt{\abs{\alpha_{0}^{-1}-\frac{k}{n}}}\right)$
and $G_{0}(\pm\gamma)=\cO\left(\sqrt{\abs{\alpha_{0}^{-1}-\frac{k}{n}}}\right)$
because
\[
\abs{z_{\mp}-z_{\pm}}=\frac{1-\lambda^{2}}{\lambda(k/n)}\sqrt{\abs{\left(k/n-\alpha_{0}\right)\left(\alpha_{0}^{-1}-k/n\right)}}.
\]
In particular $A_{0}=\cO\left(\sqrt{\abs{\alpha_{0}^{-1}-\frac{k}{n}}}\right)$
and $A_{1}\asymp1$. With these expansion at hand we are ready to conclude the proof of Proposition~\ref{thm:main}.
\begin{proof}[Proof of Proposition~\ref{thm:main}, points 2), 3), 4)]
We rely on the representation in terms of the Airy function:
\begin{enumerate}
\setcounter{enumi}{1}
\item  If $k/n>\alpha_{0}^{-1}+n^{-2/3}$ and $\lim_{n}k/n=\alpha_{0}^{-1}$
then $n^{2/3}\gamma^{2}\rightarrow+\infty$ as $n$ tends to $\infty$
and since $Ai(x)\sim\frac{1}{2x^{1/4}\sqrt{\pi}}\exp\left(-\frac{2}{3}x^{3/2}\right)$ as
$x\rightarrow+\infty$ we find on the one hand
\begin{align*}
\Abs{A_{0}\frac{Ai(n^{2/3}\gamma^{2})}{n^{1/3}}} 
 & \lesssim\frac{\left(\frac{k}{n}-\alpha_{0}^{-1}\right)^{1/4}}{n^{1/2}}\exp\left(-\frac{2}{3}n\left(\frac{k}{n}-\alpha_{0}^{-1}\right)^{3/2}\right)
\end{align*}
and since $Ai'(x)\sim-\frac{x^{1/4}}{2\sqrt{\pi}}\exp\left(-\frac{2}{3}x^{3/2}\right)$
as $x\rightarrow+\infty$ we find on the other hand
\[
\Abs{A_{1}\frac{Ai'(n^{2/3}\gamma^{2})}{n^{2/3}}}\lesssim\frac{n^{1/6}(\frac{k}{n}-\alpha_{0}^{-1})^{1/4}}{n^{2/3}}\exp\left(-\frac{2}{3}n\left(\frac{k}{n}-\alpha_{0}^{-1}\right)^{3/2}\right)
\]
which completes the proof.

\item If $k/n\in(\alpha_{0}^{-1}-n^{-2/3},\alpha_{0}^{-1}+n^{-2/3}]$
then $n^{2/3}\gamma^{2}$ is bounded in $n$ and $\abs{z_{\mp}-z_{\pm}}=\cO\left(\sqrt{\abs{\alpha_{0}^{-1}-\frac{k}{n}}}\right)$
and therefore $\frac{Ai(n^{2/3}\gamma^{2})}{n^{1/3}}A_{0}=\cO(n^{-1/3-1/3})=\cO(n^{-2/3})$
and $\frac{Ai'(n^{2/3}\gamma^{2})}{n^{2/3}}A_{1}=\cO(n^{-2/3})$.

\item If $k/n<\alpha_{0}^{-1}-n^{-2/3}$ and $\lim_{n}k/n=\alpha_{0}^{-1}$
then either $n^{2/3}\gamma^{2}$ is bounded and we refer to point (3) above, or $n^{2/3}\gamma^{2}\rightarrow-\infty$ as $n$ tends to $\infty$
and since $Ai(-x)\sim\frac{1}{x^{1/4}\sqrt{\pi}}\cos\left(\frac{2}{3}x^{3/2}-\frac{\pi}{4}\right)$
as $x\rightarrow+\infty$ we find on one hand
\[
\Abs{\frac{Ai(n^{2/3}\gamma^{2})}{n^{1/3}}A_{0}}\lesssim\frac{1}{n^{1/6}}\frac{1}{(\alpha_{0}^{-1}-\frac{k}{n})^{1/4}}\frac{\left(\alpha_{0}^{-1}-\frac{k}{n}\right)^{1/2}}{n^{1/3}}
\]
because $(\alpha_{0}^{-1}-\frac{k}{n})^{1/4}\rightarrow0$ as $n$
tends to $\infty$, and on the other hand with $Ai'(-x)\sim\frac{x^{1/4}}{\sqrt{\pi}}\sin\left(\frac{2}{3}x^{3/2}-\frac{\pi}{4}\right)$
as $x\rightarrow+\infty$ we find
\[
\Abs{\frac{Ai'(n^{2/3}\gamma^{2})}{n^{2/3}}A_{1}}\lesssim\frac{n^{1/6}(\alpha_{0}^{-1}-\frac{k}{n})^{1/4}}{n^{2/3}}.
\]
\end{enumerate}
\end{proof}
We conclude this section with the desired upper estimate on the $l_\infty^A$-norm of $(1-z^{2})b_{\lambda}^{n}$.


\begin{thebibliography}{BDFKK}


{\normalsize{\bibitem[AJ1]{AJ1} Andersson J.,}}\textit{\normalsize{ Turán's problem 10 revisited,}}{\normalsize{ preprint arXiv:math/0609271, (2008).}}{\normalsize \par}

{\normalsize{\bibitem[AJ2]{AJ2}  Andersson J.,}}\textit{\normalsize{ On some power sum problems of Turán and Erdös,}}{\normalsize{ Acta
Math. Hungar. 70:4, 305--316, (1996)}}{\normalsize \par}

{\normalsize{\bibitem[BlHa]{BlHa} Bleistein N., Handelsman R.A.,
}}\textit{\normalsize{Asymptotic Expansions of Integrals}}{\normalsize{,
Dover Publications, Inc., New York, second edition, (1986).}}{\normalsize \par}







{\normalsize{\bibitem[CFU]{CFU} Chester C., Friedman B.,  Ursell F.,
}}\textit{\normalsize{An extension of the method of steepest descents}}{\normalsize{,
Mathematical Proceedings of the Cambridge Philosophical Society 53:3, 599--611, (1957).}}{\normalsize \par}

{\normalsize{\bibitem[DS]{DS} Davies E.B. and Simon B., }}\textit{\normalsize{Eigenvalue
estimates for non-normal matrices and the zeros of random orthogonal
polynomials on the unit circle}}{\normalsize{, J. Approx. Theory 141:2,
189--213 (2006).}}{\normalsize \par}

{\normalsize{\bibitem[ER]{ER} Erdös P.,  Rényi A., }}\textit{\normalsize{A
probabilistic approach to problems of Diophantine approximation,}}{\normalsize{
lllinois J. Math. 1, 303--315, (1957).}}{\normalsize \par}

{\normalsize{\bibitem[EA]{EA} Erdélyi A., }}\textit{\normalsize{Asymptotic
representations of Fourier integrals and the method of stationary
phase}}{\normalsize{, J. Soc. Indust. Appl. Math. 3, 17--27 (1955).}}{\normalsize \par}

{\normalsize{\bibitem[FF]{FF} Foia\c{s} C., Frazho A.E., }}\textit{\normalsize{The Commutant Lifting Approach to Interpolation Problems}}{\normalsize{, Birkhäuser, Basel, (1990).}}{\normalsize \par}

{\normalsize{\bibitem[GJ]{GJ} Garnett J., }}\textit{\normalsize{Bounded
Analytic Functions}}{\normalsize{, Academic Press, New York (1981).}}{\normalsize \par}



{\normalsize{\bibitem[GMP]{GMP} Gluskin E., Meyer M., Pajor A., }}\textit{\normalsize{Zeros
of analytic functions and norms of inverse matrices}}{\normalsize{,
Israel J. Math. 87, 225--242 (1994).}}{\normalsize \par}










{\normalsize{\bibitem[MH]{MH} Montgomery H.L., }}\textit{\normalsize{Ten
lectures on the interface between analytic number theory and harmonic
analysis}}{\normalsize{, AMS, coll. CBMS Regional Conference Series in Mathematics. 84, (1994).}}{\normalsize \par}


{\normalsize{\bibitem[NF]{NF} Nagy B., Foia\c{s} C., }}\textit{\normalsize{The lifting theorem for intertwining operators and some new applications}}{\normalsize{, Indiana Univ. math. J. 20,  901--904 (1971).}}{\normalsize \par}


{\normalsize{\bibitem[NF1]{NF1} Nagy B.,  Foia\c{s} C., }}\textit{\normalsize{Commutants de certains opérateurs}}{\normalsize{, Act. Sci. Math. 29, 1--17, (1968).}}{\normalsize \par}


{\normalsize{\bibitem[NFBK]{NFBK}  Nagy B., Foia\c{s} C., Bercovici H., K\'erchy L., }}\textit{\normalsize{Harmonic Analysis of Operators on Hilbert Spaces}}{\normalsize{, Springer, (2010).}}{\normalsize \par}



{\normalsize{\bibitem[NN1]{NN1} Nikolski N., }}\textit{\normalsize{Condition
numbers of large matrices and analytic capacities,}}{\normalsize{
St. Petersburg Math. J. 17, 641--682 (2006).}}{\normalsize \par}

{\normalsize{\bibitem[NN2]{NN2}  Nikolski N., }}\textit{\normalsize{Operators,
Function, and Systems: an Easy Reading}}{\normalsize{, Vol.1, Amer.
Math. Soc. Monographs and Surveys (2002).}}{\normalsize \par}

{\normalsize{\bibitem[NN3]{NN3}  Nikolski N., }}\textit{\normalsize{Treatise
on the Shift Operator}}{\normalsize{, Springer-Verlag, Berlin etc.,
1986 (Transl. from Russian, }}\textit{\normalsize{Lekzii ob operatore
sdviga}}{\normalsize{, ``Nauja'', Moskva (1980)).}}{\normalsize \par}



{\normalsize{\bibitem[OF]{OF} Olver F.W.J., }}\textit{\normalsize{Error
bounds for stationary phase approximations}}{\normalsize{, SIAM J.
Math. Anal. 5, 19--29 (1974).}}{\normalsize \par}


{\normalsize{\bibitem[QH]{QH} Queffelec H., }}\textit{\normalsize{Sur
un théorème de Gluskin-Meyer-Pajor,}}{\normalsize{ C. R. Acad. Sci.
Paris Sér. 1 Math. 317, 155--158 (1993).}}{\normalsize \par}

{\normalsize{\bibitem[SJ]{SJ} Schäffer J.J., }}\textit{\normalsize{Norms
and determinants of linear mappings,}}{\normalsize{ Math. Z. 118,
331--339 (1970).}}{\normalsize \par}


{\normalsize{\bibitem[SO1]{SO1} Szehr O.,}}\textit{\normalsize{ Eigenvalue
estimates for the resolvent of a non-normal matrix}}{\normalsize{,
J. Spectr. Theory 4:4, 783--813 (2014).}}{\normalsize \par}

{\normalsize{\bibitem[SO2]{SO2} Szehr O.,}}\textit{\normalsize{ Spectral
methods for quantum Markov chains}}{\normalsize{, Thesis, Technical
University of Munich (2013).}}{\normalsize \par}

{\normalsize{\bibitem[SRB]{SRB} Szehr O., Reeb D., Wolf M.M.,}}\textit{\normalsize{
Spectral convergence bounds for classical and quantum Markov processes}}{\normalsize{,
Commun. Math. Phys. 333:2, 565--595 (2015).}}{\normalsize \par}


{\normalsize{\bibitem[SZ1]{SZ1} Szehr O., Zarouf R., }}\textit{\normalsize{$l_{p}$-norms
of Fourier coefficients of powers of a Blaschke factor}}{\normalsize{,
J. Analyse. Math. (2017) to appear,
ArXiv:1701.02358, (2017).}}{\normalsize \par}

{\normalsize{\bibitem[SZ2]{SZ2} Szehr O., Zarouf R., }}\textit{\normalsize{On
the asymptotic behavior of Jacobi polynomials with varying parameters}}{\normalsize{, ArXiv: 1605.02509, (2016).}}{\normalsize \par}

{\normalsize{\bibitem[TP]{TP} Turán P.,}}\textit{\normalsize{ On
a new method of analysis and its applications}}{\normalsize{, Pure
and Applied Mathematics, New-York (1984).}}{\normalsize \par}



{\normalsize{\bibitem[WR]{WR} Wong R., }}\textit{\normalsize{Asymptotic
Approximations of Integrals}}{\normalsize{, Society for Industrial
and Applied Mathematics, (2001).}}{\normalsize \par}

{\normalsize{\bibitem[ZR]{ZR} Zarouf R., }}\textit{\normalsize{Sharpening
a result by E.B. Davies and B. Simon}}{\normalsize{, C. R. Acad. Sci.
Paris Sér. 1 347:15, 939--942, (2009).}}

\end{thebibliography}
\end{document}